\documentclass[leqno,10pt,a4paper]{amsart}

\newsavebox\ltmcbox

\usepackage{multirow}
\usepackage{multicol}

\usepackage[margin=2.5cm]{geometry}
\usepackage{cite}
\usepackage{graphicx}
\usepackage{amscd}
\usepackage{amssymb}
\usepackage[usenames,dvipsnames]{color}
\usepackage{hyperref}
\hypersetup{
    colorlinks,
    linkcolor={red!50!black},
    citecolor={blue!50!black},
    urlcolor={blue!80!black}
}
\usepackage{mleftright}
\usepackage{booktabs}
\usepackage{longtable}
\usepackage{etex}
\usepackage{colortbl}
\usepackage{xypic}
\usepackage[shortlabels]{enumitem}
\usepackage{tikz,setspace}
\usepackage{textcomp}
\usetikzlibrary{arrows}
\usetikzlibrary{patterns}
\usetikzlibrary{decorations.pathreplacing}
\usetikzlibrary{chains,positioning,scopes,matrix,backgrounds,decorations.pathmorphing,decorations.pathreplacing,decorations.shapes}

\setlist[enumerate]{labelsep=*, leftmargin=1.5pc}
\setlist[enumerate]{label=\normalfont(\roman*), ref=\roman*}
\renewcommand{\emptyset}{\varnothing}
\newcommand{\orig}{\mathbf{0}}
\newcommand{\QQ}{\mathbb{Q}}
\newcommand{\ZZ}{\mathbb{Z}}
\newcommand{\PP}{\mathbb{P}}
\newcommand{\CC}{\mathbb{C}}

\newcommand{\NRR}{{N_\QQ}}
\newcommand{\MRR}{{M_\QQ}}

\DeclareMathOperator{\GL}{GL}
\DeclareMathOperator{\Gr}{Gr}

\newcommand{\calH}{\mathcal{H}}
\newcommand{\calN}{\mathcal{N}}

\newcommand{\calX}{\mathcal{X}}

\newcommand{\mcc}{\mathcal{C}}

\newcommand{\Pic}[1]{\operatorname{Pic}\mleft({#1}\mright)}
\newcommand{\Cstar}{\CC^\times}
\newcommand{\abs}[1]{\left\vert{#1}\right\vert}
\newcommand{\Hom}[1]{\operatorname{Hom}\mleft({#1}\mright)}
\newcommand{\V}[1]{\operatorname{vert}\mleft({#1}\mright)}
\newcommand{\bdry}[1]{\partial{#1}}
\newcommand{\intr}[1]{{#1}^\circ}
\newcommand{\conv}[1]{\operatorname{conv}\mleft({#1}\mright)}
\newcommand{\sconv}[1]{\operatorname{conv}\mleft\{{#1}\mright\}}
\newcommand{\scone}[1]{\operatorname{cone}\mleft\{{#1}\mright\}}
\newcommand{\Vol}[1]{\operatorname{Vol}\mleft({#1}\mright)}
\newcommand{\dual}[1]{{#1}^*}
\newcommand{\Newt}[1]{\operatorname{Newt}\mleft({#1}\mright)}
\DeclareMathOperator{\mmin}{min}
\DeclareMathOperator{\mmax}{max}
\DeclareMathOperator{\coeff}{coeff}
\newcommand{\hmin}{{h_{\mmin}}}
\newcommand{\hmax}{{h_{\mmax}}}
\newcommand{\oddrow}{\rowcolor[gray]{0.95}}
\newcommand{\evnrow}{}
\DeclareMathOperator{\spec}{Spec}
\DeclareMathOperator{\Hilb}{Hilb}
\DeclareMathOperator{\proj}{Proj}
\DeclareMathOperator{\Amp}{Amp}
\newcommand\Tstrut{\rule{0pt}{3ex}}
\newcommand\Bstrut{\rule[-1.3ex]{0pt}{0pt}}
\newcommand{\GIT}{/\!\!/}
\newcommand{\XX}[2]{X_{{#1}\text{--}{#2}}}
\newcommand{\ff}[2]{f_{{#1}\text{--}{#2}}}
\newcommand{\FF}[2]{F_{{#1}\text{--}{#2}}}

\newcommand{\Fequiv}{\stackrel{F}{\sim}}
\newtheorem{theorem}{Theorem}[section]
\newtheorem{lemma}[theorem]{Lemma}
\newtheorem{proposition}[theorem]{Proposition}
\newtheorem{corollary}[theorem]{Corollary}
\newtheorem{conjecture}[theorem]{Conjecture}

\theoremstyle{definition}
\newtheorem{definition}[theorem]{Definition}
\newtheorem{example}[theorem]{Example}
\newtheorem{remark}[theorem]{Remark}
\newtheorem{claim}[theorem]{Claim}
\newtheorem{principle}[theorem]{Principle}
\graphicspath{{images/}}
\begin{document}

\author[A.\,M.\,Kasprzyk]{Alexander Kasprzyk}
\address{School of Mathematical Sciences\\University of Nottingham\\Nottingham, NG$7$\ $2$RD\\UK}
\email{a.m.kasprzyk@nottingham.ac.uk}

\author[L.\,Katzarkov]{Ludmil Katzarkov}
\address{University of Miami and HSE University, Russian Federation, Laboratory of Mirror Symmetry.}
\email{lkatzarkov@gmail.com}

\author[V.\,Przyjalkowski]{Victor Przyjalkowski}
\address{Steklov Mathematical Institute and HSE University, Russian Federation, Laboratory of Mirror Symmetry}
\email{victorprz@mi-ras.ru, victorprz@gmail.com}

\author[D.\,Sakovics]{Dmitrijs Sakovics}
\address{IBS Center for Geometry and Physics\\Pohang University of Science and Technology\\Pohang $37673$\\Korea}
\email{dmitry85@ibs.re.kr}

\title{Projecting Fanos in the mirror}
\begin{abstract}
In the paper ``Birational geometry via moduli spaces'' by I.\,Cheltsov, L.\,Katzarkov, and V.\,Przyjalkowski
a new structure connecting toric degenerations of smooth Fano threefolds by projections was introduced; using Mirror Symmetry
these connections were transferred to the side of Landau--Ginzburg models. In the paper mentioned above
a nice way to connect of Picard rank one Fano threefolds was found. We apply this approach to all smooth Fano threefolds,
connecting their degenerations by toric basic links. In particular, we find a lot of Gorenstein toric degenerations of smooth Fano threefolds we need.
We implement mutations in the picture as well.
It turns out that appropriate chosen toric degenerations of the Fanos are given by toric basic links
from a few roots.
We interpret the relations we found in terms of Mirror Symmetry.
\end{abstract}
\maketitle
\setcounter{tocdepth}{1}
\tableofcontents
\section{Introduction}\label{section:Introduction}

One of the central topics of research  in birational geometry are Fano varieties ---  varieties  with ample anticanonical bundle.
They play a  crucial role in Minimal Model Program and present a reach geometric picture.
Fano varieties are central in  Mirror Symmetry --- many  of constructions of mirror duality are either   Calabi--Yau manifolds  or for Fano varieties.

Classification problem for smooth Fano varieties goes back to XIX century. Due to Riemann, the only Fano curve is a projective line $\PP^1$.
Pasquale del Pezzo classified smooth Fano surfaces; now they are named after him. He showed that these surfaces (with very ample anticanonical class)
are non-degenerate surfaces of degree $n$ in $\PP^n$; now we also add two trigonal examples of degrees $2$ and $1$. All these surfaces have degree at most $9$
and form an irreducible family for each degree with an exception of degree $8$ where there are two irreducible families.
Modern (that is, classical, not pre-classical) description of del Pezzo surfaces is as blow ups of general enough points on $\PP^2$ together with
a quadric surface. In other words, they are projections of anticanonically embedded $\PP^2\subset \PP^9$ from general enough points (again together with a quadric).
If we choose any points as centers of projection we get singular del Pezzo surfaces because in this case the projection may contract
lines through the point; however anyway we arrive to the same family, so smooth del Pezzo surfaces can be obtained as smoothings
of projected singular ones.

Classification of Fano threefolds is  more tricky. It was initiated  by Gino Fano and  developed later by Iskovskikh~\cite{Is77},~\cite{Is78}. (Iskovskikh gave the modern definition of Fano varieties
and named them after Fano). Soon after this Mori and Mukai, using Iskovskikh approach and Minimal Model Program, classified all smooth Fano threefolds~\cite{MM81};
it turned out to be $105$ families of them (the last one was found in 2002 in~\cite{MM03}). There is no classification in higher dimensions. By Koll\'ar--Miyaoka--Mori, there is a finite number
of families of Fano varieties in any given dimension. However already in dimension $4$ it is expected that the number of  families of Fano varieties is very big.

Unlike the two-dimensional case, there is no structure in the list of Fano threefolds (see~\cite{IP99}) relating one with each other systematically.
An approach to get such structure is given in~\cite{CKP13}. The idea  behind the approach is the following.
Similarly to the two-dimensional case, one can relate all Fano variety themselves to  some "specific" varieties (not necessary smooth).
A class of simple relations is (similarly to the surface case) projections from singular points, tangent spaces to smooth points,
lines, and conics.

Finally, to put the problem on the combinatorial level, we choose toric
Fano varieties as the ``specific''  varieties.
We call the simple projections between toric varieties \emph{toric basic links}.
Thus one can construct, in terms of spanning polytopes the needed  projections. An example of nice subtree
in the projections tree relating Picard rank one Fano varieties (Figure~\ref{figure:del Pezzo tree}) is found in~\cite{CKP13}.
Moreover, one can implement mutations in the picture (see Section~\ref{sec:mutation}), that is deformations from one toric degeneration to another.
In the present  paper we study projections systematically for all Fano threefolds.
In particular, we prove the following.
Given a toric variety $T$ let us call a projection in an anticanonical embedding from tangent space to invariant smooth point, invariant cDV point, or an invariant smooth line F-projection.

\begin{theorem}[Theorem \ref{thm:proj:graph_proj-mut}]
\label{theorem: main 1}
 Given any smooth Fano threefold $X$, there exists a Gorenstein toric degeneration of $X$ that can be obtained by a sequence of mutations and F-projections
 from a toric degeneration of one of $15$ smooth Fano threefolds (see Table~\ref{table:proj:rootFanos}). 
 The directed graph connecting all Fano varieties with very ample anticanonical class via the projections and mutations is presented in Table~\ref{table:proj:choice}.
Each of toric degenerations we use can be equipped with a toric Landau--Ginzburg model.
\end{theorem}

\begin{theorem}[Theorem \ref{hyi}]
\begin{itemize}
\label{theorem: main 2}
 \item[(i)]
For any smooth Fano threefold with very ample anticanonical class there is its Gorenstein toric degeneration such that all these degenerations are connected by sequences of F-projections. The directed graph of such projections 
can be chosen as a union of $15$ trees with roots shown in Table~\ref{table:proj:rootFanos}.
The directed graph connecting all Fano varieties with very ample anticanonical class via the projections and mutations is presented in Figure~\ref{fig:graph:projections:i}.
Each of the toric degenerations can be equipped with a toric Landau--Ginzburg model. 

\item[(ii)]
For any smooth Fano threefold with very ample anticanonical class there is its toric degeneration such that all these degenerations are connected by sequences of projections in the anticanonical embedding with toric centres which are either tangent spaces to smooth points, or
cDV points, or smooth lines, or smooth conics. The directed (sub)graph of such projections
connecting degenerations of all smooth Fano threefolds
can be chosen to have five roots which are: $\PP^3$, $\mathbb{P}(\mathcal{O}_{\mathbb{P}^2}\oplus\mathcal{O}_{\mathbb{P}^2}(2))$,
quadric threefold, $\PP^1\times\PP^1\times \PP^1$, and~$\PP^3$ blown up in a line.
The directed graph connecting all Fano varieties via the projections and mutations is presented in Figure~\ref{fig:graph:projections:ii}.
Each of the toric degenerations can be equipped with a toric Landau--Ginzburg model.
\end{itemize}
\end{theorem}
(Note that $\PP^3$ blown up in a line is not a projection of $\PP^3$ since a line in anticanonical embedding has degree $4$.)

As the theorems suggest, the toric degenerations providing basic links correspond to toric Landau--Ginzburg models
(that is Laurent polynomials related to toric degenerations and representing mirrors, see below), and the theorems shows that 
two Fano varieties are related by a toric basic link if their toric Landau--Ginzburg models
are closely related too.

The idea of presenting Landau--Ginzburg models dual to Fano varieties or dual to varieties close to Fano
as Laurent polynomials apparently is going back to Givental.
In~\cite{Gi97} he suggested a Landau--Ginzburg model for a smooth toric variety as a complex torus with a complex-valued function (superpotential)
represented by a Laurent polynomial with support at a fan polytope of the toric variety.
His construction was generalized to varieties admitting nice toric degenerations;
in this case one associate the Laurent polynomial with the fan of the toric degeneration of the Fano variety.
This idea is going back to Batyrev--Borisov's approach to mirror duality for toric varieties
as duality of their polytopes (see~\cite{BB96}) and good deformational behavior of Gromov--Witten
invariants, see~\cite{Bat04} and references therein.
In this spirit Eguchi--Hori--Xiong for Grassmannians (see~\cite{EGH97}) and Batyrev--Ciocan-Fontanine--Kim--van Straten
for partial flag varieties (see~\cite{BCFKS97} and~\cite{BCFKS98})
constructed Laurent presentations for Landau--Ginzburg for Grassmannians.

The crucial part of the duality between Fano varieties and Landau--Ginzburg models in this approach
is an identification of (a part of) Gromov--Witten theory for the Fano varieties
and periods of the dual one-dimensional family. If the total space of the family is a complex torus, then
the Landau--Ginzburg model, as we mentioned above, is (in some basis) represented by a Laurent polynomial.
In this case the main period of the family is a generating series for constant terms of powers of Laurent polynomials,
see Section~\ref{sec:toric_LG_models} again. In~\cite{Gi97} Givental proved the coincidence of the two series for smooth
toric varieties; the same was done in~\cite{BCFKS97} and~\cite{BCFKS98} for Grassmannians and partial flag varieties.

In~\cite{Gi97} Givental suggested an approach to constructing Landau--Ginzburg model for (almost) Fano complete
intersection in toric varieties. This approach was generalized to complete intersections in Grassmannians and (partial)
flag varieties in~\cite{BCFKS97} and~\cite{BCFKS98}, see also~\cite{BCFK03}.
The output of Givental's construction is not a complex torus with a function but a complete intersection in a torus
with a function. These Landau--Ginzburg models satisfy the Gromov--Witten--period condition via Quantum Lefshetz Theorem,
which enables one to pass from Gromov--Whitten theory of a variety to Gromov--Whitten theory of an ample (Fano) hypersurface therein.

The natural idea is to realize birationally  Landau--Ginzburg models as Laurent polynomials.
This was done in~\cite{Prz07},~\cite{Prz09},~\cite{ILP11},~\cite{Prz17},~\cite{CCGK13} for smooth Fano threefolds and complete
intersections in projective spaces.

In the case of  Fano  complete intersections (Picard rank one)  it was proven that the resulting Laurent polynomials
are related to toric degenerations.
This led to the idea that Laudau--Ginzburg models presented as Laurent polynomials should be assigned with toric
degenerations  used in the generalization of Givental's approach.
Another concept is the Compactification Principle which says that correctly chosen Laurent Landau--Ginzburg
model admits a fiberwise compactification to a family of compact Calabi--Yau varieties such that
this family satisfy (an algebraic part of) Homological Mirror Symmetry. These two concepts, together
with the initial concept of Gromov--Witten--period coincidence form  the central ideas  on which  the present  paper is based.
The above ideas were initiated by  Golyshev, who has suggested a program of finding (toric) Landau--Ginzburg model
by guessing Laurent polynomials having prescribed constant terms and, hence, periods.
The first results were obtained for smooth rank one Fano threefolds in~\cite{Prz07} and~\cite{Prz09};
later in~\cite{ILP11} the proof of their toricity was completed.
These results clarified the connection with toric degenerations to (possibly very singular) toric varieties.
The specific simple form of found Laurent polynomials leads to binomial principle claimed in~\cite{Prz09}.
This principle states that coefficients on facets of the Newton polytope of the Laurent polynomial correspond to
binomial coefficients of a power of a sum of independent variables. Surprisingly this principle
covers most of (but not all) smooth Fano threefolds: most of them have toric degenerations with cDV singularities,
which means that integral points of the Newton polytope for Laurent polynomial are the origin and points lying on edges.
This gives an algorithm for finding Landau--Ginzburg models.

Binomial principle was generalized to Minkowski principle in~\cite{CCGGK13}. It relates coefficients
of the Laurent polynomial with Minkowski decomposition(s) of facets of its Newton polytope
to particular elementary summands. Moreover, in loc. cit. all canonical polytopes
that are Newton polytopes of Minkowski Laurent polynomials were found and all $J$-series
of smooth Fano threefolds were computed. This gives, for any smooth Fano threefold, a Laurent polynomial (which is not unique)
satisfying period condition. In this paper we establish toric degeneration condition as well, see Section~\ref{sec:degen_intro}.

\begin{theorem}
\label{theorem:degenerations}
Let $T$ be a Gorenstein toric variety appeared in Theorems~\ref{theorem: main 1} and~\ref{theorem: main 2}.
Then $T$ is a toric degeneration of corresponding Fano threefold.
\end{theorem}

\begin{remark}
The assertion of Theorem~\ref{theorem:degenerations} holds for much larger class of Gorenstein toric varieties, see Appendix~\ref{sec:appendix}.
\end{remark}

As we have mentioned, Laurent polynomial, as a mirror dual to Fano, is not unique.
However it's Calabi--Yau compactification is unique. Indeed under mild natural condition, holds for rank one Fano threefolds,
see~\cite{ILP11} and~\cite{DHKLP}. This means that Landau--Ginzburg models are birational over the base field $\CC(x)$.
In other words, corresponding Laurent polynomials differ by mutations. It is proven in~\cite{ACGK12} that all
Laurent polynomials with support on reflexive polytopes that produce the same period differ by (a sequence of) mutations.
So they have a common log Calabi--Yau compactification. This suggests the last, at the moment, and the strongest
concept of assigning a Laurent Landau--Ginzburg model to a reflexive polytope, the maximal mutational principle.
It is described in details in Section~\ref{sec:mutation}. 

The above findings can be interpreted categorically.
We propose the following.

\begin{conjecture}
The moduli spaces of Landau--Ginzburg models (defined in \cite{DKK13}) for directed graphs of
Fano varieties from Theorem
~\ref{theorem: main 2} are contained in each other
with the top Landau--Ginzburg models contained in the ones obtained by projections.
\end{conjecture}

In such a way the behaviour of Landau--Ginzburg models for three-dimensional Fano varieties
is very similar to the behaviour of Landau--Ginzburg models of
del Pezzo surfaces.

We lift this conjecture to further categorical levels.
As a consequence of the connection of curve complexes and stability
conditions it was noticed in \cite{DK} that stability conditions
should behave well in families. Later on, the following theorem was
proven by Haiden, Katzarkov, Kontsevich, and Pandit in \cite{HKKP}.
Below we  use the definition of stability conditions given by
Bridgeland.  We give the most general version of the statement.
After that we will explain the connection with our situation.
In what follows we give a categorical description of a family of
hyperplane sections.  We use the language of comonads.
Here the category $\mcc_{\text{special}}$ is the analogue of a singular hyperplane
section and the category $\mcc_{\text{general}}$ is the analogue of a general section. The category $\mcc_0$
is the global family.

\begin{theorem}
Consider the following data.
\begin{enumerate}
\item A category   $\mcc_{\text{special}}$ which   is an $( \infty, 1
)$-category.

\item A stability condition on $\mcc_{\text{special}}$.

\item   A comonad $T$ on $\mcc_{\text{special}}$ such that $Cone(T \to
Id) =[2]$.

\end{enumerate}


Let $\mcc_0$ be a category of comodules corresponding to $\mcc_{\text{special}}$ and $T$.
There is a functor $\mcc_{\text{special}} \to\mcc_0$.
We  define $\mcc_{\text{gen}}=  \mcc_0 / \mcc_{\text{special}} $.

In the situation above there exists a stability condition on
$\mcc_{\text{gen}}$
such that its  central charge and its phase are lift from
a central charge and a phase on   $\mcc_{\text{special}} $.

\end{theorem}

Applied to our situation the above theorem suggests the following.

\begin{enumerate}
\item Stability conditions of Fano varieties can be obtained from
stability conditions of (singular) toric varieties.
Indeed the  last ones have exceptional collections and moduli spaces
of stability conditions are  easier to understand.

\item Stability conditions of Calabi--Yau varieties can be obtained from stability
conditions of Fano  manifolds via Tyurin degenerations.
\end{enumerate}

Combining these facts with the  finding of the present   paper suggest the following.

\begin{conjecture}
The moduli spaces of stability conditions (defined by Bridgeland) for directed graphs of Fano varieties in Theorem
~\ref{theorem: main 2} are contained in each other
with the top moduli spaces of stability conditions  contained in the
ones obtained by projections.
\end{conjecture}

%
%
%
%
%
%

The above observations suggest that obtaining via degenerations
stability condition for one of three-dimensional Fano varieties leads to
computing stability conditions for all of them.
In a similar fashion  we propose  that Apery  constants
(defined in  \cite{GOL})
for all these Fano varieties are connected with each other.
We expect that similar behaviour of Fano varieties  extends to
high dimensions --- of course diagrams are going to be enormous.

\medskip

The paper is organized as follows.
In Section~\ref{sec:basic_links} we recall results from~\cite{CKP13} and define toric basic links relating Fano threefolds.
In Section~\ref{sec:toric_LG_models} we define toric Landau--Ginzburg models associated with toric Fano threefolds.
Toric basic links can be interpreted as their transformations. In Section~\ref{sec:mutation} we define mutations between toric
Landau--Ginzburg models, that is relative birational transformations between them. They correspond to deformations of toric degenerations
of given Fano threefolds and they can be implemented to the toric basic links graph.
In Section~\ref{sec:degen_intro} we study toric degenerations of Fano threefolds.
In Section~\ref{sec:minimality} we describe the directed graph of reflexive polytopes.
In Section~\ref{sec:projections_algorithm} we describe the algorithm we use to compute the projections graph.
In Section~\ref{sec:projection_trees} we compute the directed (sub)graph of projections relating smooth Fano threefolds; roots of the graphs are several particular Fano threefolds.
Finally in Appendices we present the data which is the output of our construction. In Appendix~\ref{sect:minGraph} we present the appropriate projection directed (sub)graphs.
In Appendix~\ref{sec:appendix} we present Gorenstein toric degenerations of Fano threefolds.
In Appendices~\ref{app:compdims} and~\ref{app:tsdims} we present data related to the toric degenerations.

\subsection{Notation}\label{sec:notation}
Smooth del~Pezzo threefolds, that is, smooth Fano threefolds of index two, we denote by $V_d$, where $d$ is the degree with respect to a generator of the Picard group; the single exception is the quadric, which we denote by $Q$. Fano threefold of Picard rank one, index one, and degree $d$, we denote by $X_d$. The remaining Fano threefolds we denote by $\XX{k}{n}$, where $k$ is the Picard rank and $n$ is its number according to~\cite{IP99}. When $k=4$ the numbers $n$ differ from the identifiers in Mori--Mukai's original classification~\cite{MM81} due to the `missing' rank four Fano $\XX{4}{2}$~\cite{MM03}, which has been placed in the appropriate position within the list.

To any reflexive polytope $P\subset\NRR$ we associate the Gorenstein toric Fano variety $X_P$ whose fan $\Sigma$ in the lattice $N$ is generated by taking the cones over the faces of $P$. We call $\Sigma$ the \emph{spanning fan} of $P$. The moment polytope is denoted by either $\dual{P}\subset\MRR$ or by $\Delta\subset\MRR$, and is dual to $P$. We identify Gorenstein toric Fano varieties by the corresponding reflexive polytope $P$, numbered from $1$ to $4319$ by the \emph{Reflexive ID} as given by the online Graded Ring Database~\cite{GRDb}. This agrees with the order of the output from the software~\textsc{Palp}~\cite{KS04}, developed by Kreuzer--Skarke for their classification~\cite{KS98}, and with the databases used by the computational algebra systems~\textsc{Magma} and~\textsc{Sage}, the only complication being whether the numbering starts from $1$ (as done here) or from $0$.

The Laurent polynomial associated to a Fano threefold of Picard rank $k$ and number $m$ we denote by $\ff{k}{m}$. The toric variety whose spanning fan is generated by the Newton polytope of $\ff{k}{m}$ we denote by $\FF{k}{m}$. When appropriate, we may refer to the period period sequence $\pi_f(t)$ of a Laurent polynomial $f$ by its \emph{Minkowski ID}, an integer from $1$ to $165$. The Minkowski IDs are defined in~\cite[Appendix~A]{ACGK12}, used in~\cite{CCGK13}, and can be looked-up online at~\cite{GRDb}.

\subsection{The use of computer algebra and databases}\label{sec:computer_algebra}
Because of the large number of $3$-dimensional Gorenstein toric Fano varieties, several results are derived with the help of computational algebra software and databases of classifications such as~\cite{GRDb}. Computer-assisted rigorous proofs play an increasingly important role as we move from surfaces to threefolds, and will become an essential mathematical technique if we ever hope to progress to the systematic study of fourfolds and Kreuzer--Skarke's massive classification~\cite{KS00} of $473\,800\,776$ reflexive polytopes in dimension four.

We highlight our use of computer algebra. In~\S\ref{sec:degen_intro}, step~\eqref{stepone}, we make use of \textsc{Magma}~\cite{magma} in order to compute additional toric models, besides those arising from the Minkowski polynomials as classified in~\cite{ACGK12}. In~\S\ref{sec:degen_intro}, step~\eqref{stepthree}, we use \textsc{Macaulay2}~\cite{M2} to compute the dimension of the tangent space for each of the Gorenstein toric Fano threefolds. The software \textsc{Topcom}~\cite{TOPCOM} is used in~\S\ref{sec:triang} to search for reflexive polytopes with appropriate boundary triangulations. Finally, in~\S\ref{sec:projection_trees} we make use of several computer programs that rely on \textsc{Palp}~\cite{KS04} in order to build and manipulate the relevant projection graphs, as well as to further explore the effects of using different combinations of allowed projections and mutations. We emphasize that although any particular example can be worked by hand, the number of cases under consideration means that the only practical way to ensure accuracy is to employ the use of a computer.

\subsection{Acknowledgements}
We are very grateful to Nathan IIten for his help with the calculations in
Section~\ref{sec:degen_intro}. Without his generous contributions, this paper would not
have been possible.

\smallskip

Alexander Kasprzyk is supported by EPSRC Fellowship EP/N022513/1. 
Ludmil Katzarkov  was supported by Simons research grant, NSF DMS 150908, ERC Gemis, DMS-1265230, DMS-1201475, OISE-1242272 PASI, Simons collaborative Grant --- HMS,
Simons investigator grant --- HMS; he is partially supported by Laboratory of Mirror Symmetry NRU HSE, RF Government grant, ag. \textnumero~14.641.31.0001.
Victor Przyjalkowski was partially supported by Laboratory of Mirror Symmetry NRU HSE, RF Government grant, ag. \textnumero~14.641.31.0001.
He is Young Russian Mathematics award winner and would like to thank its sponsors and jury.

\section{Basic links}\label{sec:basic_links}
\begin{definition}
Fano variety $X$ is called \emph{Gorenstein} if its anticanonical class $-K_X$ is a Cartier divisor. Such a variety is called \emph{canonical} (or is said to have \emph{canonical singularities}) if for any resolution $\pi\colon\widetilde{X}\rightarrow X$ the relative canonical class $K_{X}-\pi^*K_{\widetilde{X}}$ is an effective divisor.
\end{definition}

\begin{definition}
A threefold singularity $P$ (not necessary isolated) is called \emph{cDV} if its general transversal section has a du~Val singularity at $P$.
\end{definition}

\noindent
The only canonical Gorenstein Fano curve is $\PP^1$. Canonical Gorenstein Fano surfaces are called del~Pezzo surfaces; they are given by the quadric surface, $\PP^2$, and the blow-up of $\PP^2$ in at most eight points in general position, along with the degenerations of these smooth surfaces. Canonical Gorenstein Fano threefolds are not yet classified, although partial results can be found in~\cite{Mu02,CPSh05,Pr05,JR06,Ka09}. Smooth Fano threefolds were classified by Iskovskikh~\cite{Is77,Is78} and Mori--Mukai~\cite{MM81,MM03}: there are $105$ deformation classes, of which $98$ have very ample anticanonical divisor $-K_X$.

Let $\varphi_{\abs{-K_X}}\colon X\rightarrow\PP^{g+1}$ be a map given by $\abs{-K_X}$. Then one of the following occurs:
\begin{enumerate}
\item
$\varphi_{\abs{-K_X}}$ is not a morphism, that is $\mathrm{Bs}\abs{-K_X}\neq \emptyset$, and all such $X$ are found in~\cite{JR06};
\item
$\varphi_{\abs{-K_X}}$ is a morphism but not an embedding, the threefold $X$ is called \emph{hyperelliptic}, and all such $X$ are found in~\cite{CPSh05};
\item
$\varphi_{\abs{-K_X}}$ is an embedding and $\varphi_{\abs{-K_X}}(X)$ is not an intersection of quadrics --- then the threefold $X$ is called \emph{trigonal}, and all such $X$ are found in~\cite{CPSh05};
\item
$\varphi_{\abs{-K_X}}(X)$ is an intersection of quadrics.
\end{enumerate}
In particular an anticanonically embedded Fano variety is either trigonal or an intersection of quadrics. The varieties that cannot be anticanonically embedded are classified: only few of them can be smoothed. In this paper we focus on anticanonically embedded threefolds.

\subsection{The del~Pezzo surfaces}\label{sec:del_Pezzo_surfaces}
In the 1880s Pasquale del~Pezzo considered the surfaces of degree $n$ in $\PP^n$; in other words, surfaces with a very ample anticanonical class. The modern definition of \emph{del~Pezzo surfaces} as ones with ample anticanonical class adds blow-ups of $\PP^2$ in seven and eight general points to del~Pezzo's initial list of surfaces.

When classifying Fano varieties, we in fact classify their moduli spaces or components of their deformation spaces. This means that it is natural to consider degenerations of Fano varieties. Given one point on each moduli space together with its deformation information one can reconstruct a classification of Fano varieties. For del~Pezzo surfaces this means that we allow birational transformations related to points in both general and non-general position. Of course, in this case we may get singular surfaces. Following the work of del~Pezzo, let us consider anticanonically embedded surfaces. A blow-up of a general point is simply a projection from this point. A similar transformation related to a non-general point is a blow-up of this point followed by taking the anticanonical model. The anticanonical map contracts $-2$-curves that can appear after the blow-up, i.e.~strict transforms of the lines passing through the original point. Thus these transformations are nothing but projections in the anticanonical embedding.

Projections from smooth points relate (possibly non-general) anticanonically embedded del~Pezzo surfaces, starting at either $\PP^2=S_9\subset\PP^9$ or the quadric $Q$, and finishing with the cubic $S_3\subset\PP^3$:
\[
\xymatrix{
 & & Q \ar@{-->}[d]^-{\pi_{8'}} & & &\\
\PP^2=S_9\ar@{-->}[r]^-{\pi_9} & S_8 \ar@{-->}[r]^-{\pi_8} & S_7\ar@{-->}[r]^-{\pi_7} & S_6\ar@{-->}[r]^-{\pi_6}
& S_5\ar@{-->}[r]^-{\pi_5} & S_4\ar@{-->}[r]^-{\pi_4} & S_3.}%
\]

\noindent
Our projections terminated on a cubic: further projections are non-birational. The reason for this is that cubic surface is trigonal, and blowing up a point produces a hyperelliptic surface whose anticanonical map is a double covering. Thus we want to consider birational transformations
\[
\xymatrix{&\widetilde{X}\ar@{->}[dl]_{\alpha}\ar@{->}[dr]^{\varphi_{|-K_{\widetilde{X}}|}}&\\%
X\ar@{-->}[rr]_{\pi_{i}}&&X',}
\]
where $\alpha$ is a blow-up, such that:
\begin{enumerate}
\item\label{item:birational}
the map $\varphi_{\abs{-K_{\widetilde{X}}}}$ is birational;
\item\label{item:fano}
the variety $X'$ is Fano.
\end{enumerate}
Condition~\eqref{item:birational} is satisfied by being $X$ an intersection of quadrics; in this case, $\widetilde{X}$ is at most trigonal. Via a careful choice of centres of the blow-ups, condition~\eqref{item:fano} can also be satisfied. Considering projections in the surface case, we need to include smooth points in the set of admissible centres; \emph{a posteriori} we see that this is sufficient.

In order to simplify the model surface picture, we make one final choice: the particular points of the deformation spaces of del~Pezzo surfaces that we want to relate. We wish to make use of the toric del~Pezzo surfaces. Our main motivation for this choice comes from our desire to exploit the relation, via toric Landau--Ginzburg models, with mirror symmetry; the relative ease of working with toric varieties; and the classifications of toric varieties. In this case, we insist that the centres of blow-ups should also be toric points. Thus we obtain to the following definition:

\begin{definition}\label{def:basicLink:2D}
A diagram
\[
\xymatrix{&\widetilde{S}_{i}\ar@{->}[dl]_{\alpha_{i}}\ar@{->}[dr]^{\varphi_{|-K_{\widetilde{S}_i}|}}&\\%
S_i\ar@{-->}[rr]_{\pi_{i}}&&S_{i-1},}
\]
where $\alpha_i$ is a blow-up of a smooth point and $i\geq 4$ is called a \emph{basic link} between del~Pezzo surfaces. If the varieties and the centre of the blow-up are toric then the basic link is called \emph{toric}.
\end{definition}

There are $16$ toric del~Pezzo surfaces, corresponding to the $16$ reflexive polygons~\cite{Bat85,Rab89}. All possible toric basic links between them are drawn in Figure~\ref{figure:del Pezzo tree}. Any chain of toric basic links from $\PP^2$ to the toric cubic, plus an appendix with the quadric, gives us the classification of del~Pezzo surfaces in the sense discussed above.

\begin{figure}[ht]
\centering
\centerline{
\xymatrix{
\rule{60pt}{0pt}&\rule{77pt}{0pt}&&\rule{77pt}{0pt}&\rule{60pt}{0pt}\\
&&\mbox{$\phantom{\PP^2}$\includegraphics[scale=0.65]{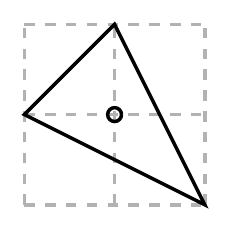}\raisebox{20pt}{$\PP^2$}}\ar[d]&&\\
&\includegraphics[scale=0.65]{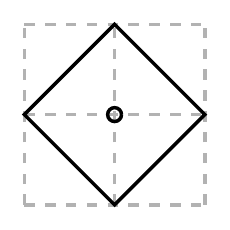}\ar[d]&
\includegraphics[scale=0.65]{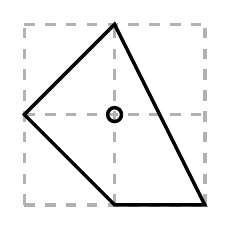}\ar[dr]\ar[dl]&
\includegraphics[scale=0.65]{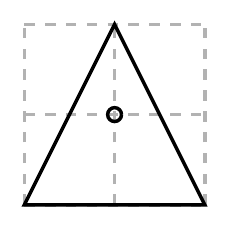}\ar[d]&\\
&\includegraphics[scale=0.65]{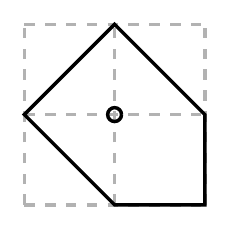}\ar[dl]\ar[d]\ar[drr]&
&\includegraphics[scale=0.65]{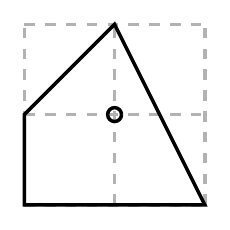}\ar[dll]\ar[dr]\ar[d]&\\
\includegraphics[scale=0.65]{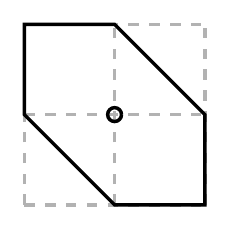}\ar[dr]&
\includegraphics[scale=0.65]{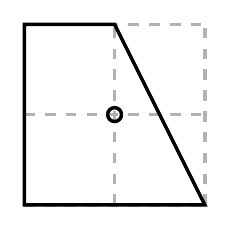}\ar[d]\ar[drr]&
&\includegraphics[scale=0.65]{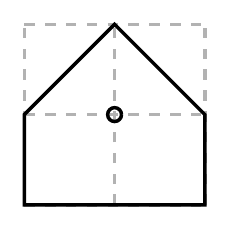}\ar[d]\ar[dll]&
\includegraphics[scale=0.65]{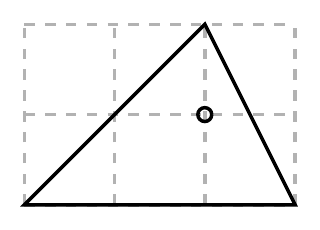}\ar[dl]\\
&\includegraphics[scale=0.65]{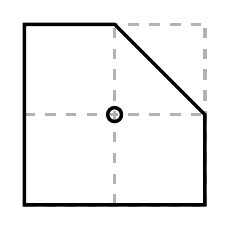}\ar[d]\ar[dr]&
&\includegraphics[scale=0.65]{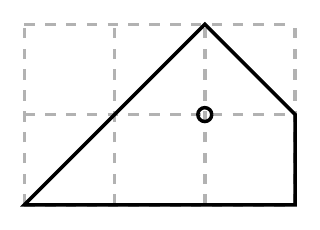}\ar[d]\ar[dl]&\\
&\includegraphics[scale=0.65]{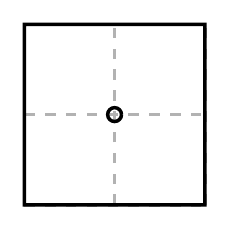}&
\includegraphics[scale=0.65]{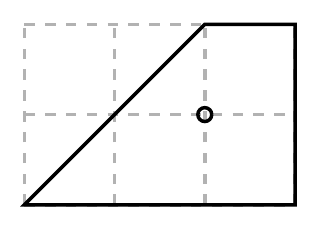}\ar[d]&
\includegraphics[scale=0.65]{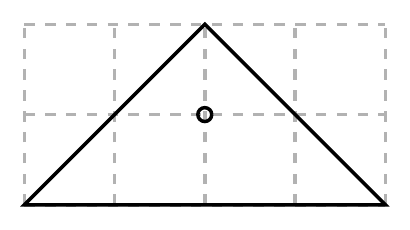}&\\
&&\mbox{$\phantom{S^3}$\includegraphics[scale=0.65]{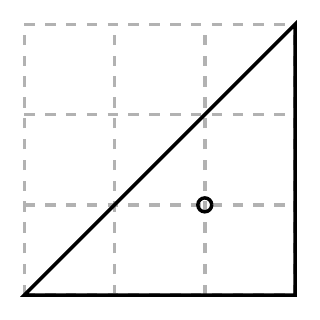}\raisebox{20pt}{$S_3$}}&&\\
}
}
\caption{The toric del~Pezzo directed graph, starting with $\PP^2$ at the top of the diagram, and working down through basic links to the toric cubic $S_3$ at the bottom of the diagram.}\label{figure:del Pezzo tree}
\end{figure}

\subsection{Threefold case}\label{sec:threefold_case}
The situation described in~\S\ref{sec:del_Pezzo_surfaces} changes dramatically when moving to three dimensions. Iskovskih~\cite{Is77,Is78} and Mori--Mukai~\cite{MM81,MM03} classified the smooth Fano threefolds around 1980, a century after del~Pezzo's work in $2$-dimensions. It had already observed by Iskovskih--Manin~\cite{IM71} that, unlike the $2$-dimensional case, not all of the smooth Fano threefolds are rational. Thus, if the basic links considered are required to be birational, there is no hope of producing a direct analogue of Figure~\ref{figure:del Pezzo tree} for smooth $3$-dimensional Fano varieties. This can be rectified by considering not the smooth Fano variety $X$ itself, but the toric degenerations of $X$ (and corresponding basic links). When $X$ is very ample, these degenerations are toric Fano threefolds with Gorenstein singularities. Reid has shown~\cite[Corollary~3.6]{Rei80} that Gorenstein toric varieties have at worst canonical singularities, and Batyrev has shown~\cite{Bat81} that the Gorenstein toric Fano varieties of dimension $n$ are equivalent, in a precise sense arising from the combinatorics of toric geometry, to the $n$-dimensional reflexive polytopes.

\begin{definition}\label{defn:reflexive}
Let $N\cong\ZZ^n$ be a lattice of rank $n$, and let $P\subset\NRR:=N\otimes_\ZZ\QQ$ be a convex lattice polytope of maximum dimension. That is, the vertices $\V{P}$ of $P$ are points in $N$, and the dimension of the smallest affine subspace containing $P$ is equal to the rank of $N$. We say that $P$ is \emph{reflexive} if the \emph{dual} (or \emph{polar}) polyhedron
\[
\dual{P}:=\{u\in\MRR\mid u(v)\geq -1\text{ for all }v\in P\},\qquad\text{ where }M:=\Hom{N,\ZZ},
\]
is a lattice polytope in $M$.
\end{definition}

\noindent
The $3$-dimensional reflexive polytopes were classified by Kreuzer--Skarke~\cite{KS98}; up to $\GL_3(\ZZ)$-equivalence there are $4319$ cases. As noted in~\S\ref{sec:notation} we will blur the distinction between a reflexive polytope $P\subset\NRR$ and the corresponding toric Fano threefold $X_P$, and both will often be referred to by their Reflexive ID~\cite{GRDb}.

We now need to produce a generalisation of the $2$-dimensional basic links in Definition~\ref{def:basicLink:2D} that would naturally connect the Gorenstein toric Fano threefolds. Let $X$ be such a threefold. Following~\cite[\S2.5]{CKP13}, let $Z$ be one of:
\begin{enumerate}
\item
a smooth point of $X$;
\item
a terminal cDV point of $X$;
\item
a line on $X$ not passing through any non-cDV points.
\end{enumerate}
Let $\alpha:\tilde{X}\rightarrow X$ be the blow-up of the ideal sheaf of the subvariety $Z\subset X$, and let $\beta:\tilde{X}\rightarrow X'$ be the morphism defined by $\left|-K_{\tilde{X}}\right|$, making $\beta$ an embedding. Then (see~\cite[Lemma~2.2]{CKP13}):

\begin{proposition}
The morphism $\beta$ is birational and $X'$ is a Fano threefold with Gorenstein singularities.
\end{proposition}

\begin{definition}\label{def:basicLink:3D}
Using the morphisms $\alpha$ and $\beta$ as defined above, consider the commutative diagram:
\[
\vcenter{\xymatrix{
&\widetilde{X}\ar@{->}[dl]_{\alpha}\ar@{->}[dr]^{\beta}&\\
X\ar@{-->}[rr]_{\pi}&&X'
}}
\]
We call $\pi:X\dashrightarrow X'$ a \emph{basic link} between the threefolds $X$ and $X'$. We denote individual types of basic links as:
\begin{enumerate}
\item
$\Pi_\mathrm{p}$ if $Z$ is a smooth point;
\item
$\Pi_\mathrm{dp}$ (or $\Pi_\mathrm{o}$, or $\Pi_\mathrm{cDV}$) if $Z$ is a double point (or, respectively, an ordinary double point, or a non-ordinary double point);
\item
$\Pi_\mathrm{l}$ if $Z$ is a line.
\end{enumerate}
If all the varieties in question are toric, we call $\pi$ a \emph{toric basic link}.
\end{definition}

\noindent
In all of these cases, $\pi$ can be naturally seen as a projection of $X$: if $Z$ is a smooth point, then $\pi$ is the projection from the projective tangent space of $X$ at $Z$, and in all other cases it is the projection of $X$ from $Z$ itself. We call $X$ the \emph{root} of the projection $\pi$. Under certain additional assumptions, it is possible to similarly define basic links for $Z$ being a higher-degree curve. For example, we let $\Pi_\mathrm{c}$ denote the basic link in the case when $Z\subset X$ is a conic curve; see~\cite{CKP13} for the definition. Since these more-general basic links are less natural, we deliberately try to avoid them in our main calculation. We will refer to projections from points and lines as the \emph{allowed} projections.

\begin{example}[{\!\!\cite[\S 2.6]{CKP13}}]
Similarly to the $2$-dimensional case, we can begin with $\PP^{3}$ (or $Q^3$) and start applying toric basic links to it and to the varieties we get as a result. We would expect to obtain, up to degeneration, all (or almost all) the very ample smooth Fano threefolds. Since the basic links can be seen as projections, we can formulate this as a directed graph with vertices corresponding to smooth Fano threefolds, up to degeneration, and arrows corresponding to the projections from a degeneration of one variety to that of another. For example, one can see a small piece of such a graph in Figure~\ref{fig:FanoSnake}.
\end{example}

\begin{figure}[ht]
\centering
 \[
\centerline{
 \xymatrix{
 & & {Q^3} \ar@{-->}[d]^-{\Pi_\mathrm{p}} & & & & & & \\
 {\PP^{3}} \ar@{-->}[d]^-{\Pi_\mathrm{p}} & & \XX{2}{30} \ar@{-->}[d]^-{\Pi_\mathrm{l}} & & & & & & \\
 \XX{2}{35} \ar@{-->}[d]^-{\Pi_\mathrm{p}} & & \XX{3}{23} \ar@{-->}[d]^-{\Pi_\mathrm{o}} & & & & & & \\
 \XX{2}{32} \ar@{-->}[d]^-{\Pi_\mathrm{p}} \ar@{-->}[r]^-{\Pi_\mathrm{c}} &
 \XX{3}{24} \ar@{-->}[r]^-{\Pi_\mathrm{o}} & \XX{4}{10}
 \ar@{-->}[r]^-{\Pi_\mathrm{c}} & \XX{4}{7} \ar@{-->}[r]^-{\Pi_\mathrm{c}} &
 \XX{3}{12} \ar@{-->}[r]^-{\Pi_\mathrm{o}} & \XX{3}{10} \ar@{-->}[r]^-{\Pi_\mathrm{o}} & \XX{4}{1} \ar@{-->}[r]^-{\Pi_\mathrm{o}} & {V_{22}} \ar@{-->}[r]^-{\Pi_\mathrm{o}} & \XX{2}{13} \ar@{-->}[d]^-{\Pi_\mathrm{o}} \\
 {B_5} \ar@{-->}[d]^-{\Pi_\mathrm{p}} & & & & & & & & {V_{18}} \ar@{-->}[d]^-{{\Pi_\mathrm{cDV}}}\\
 {B_4} \ar@{-->}[d]^-{\Pi_\mathrm{p}} & & {V_4} & {V_6} \ar@{-->}[l]_-{{\Pi_\mathrm{cDV}}} & {V_8} \ar@{-->}[l]_-{{\Pi_\mathrm{cDV}}} & {V_{10}} \ar@{-->}[l]_-{{\Pi_\mathrm{cDV}}} & {V_{12}} \ar@{-->}[l]_-{\Pi_\mathrm{o}} & {V_{14}} \ar@{-->}[l]_-{\Pi_\mathrm{o}} & {V_{16}} \ar@{-->}[l]_-{\Pi_\mathrm{o}} \\
 {B_3} \ar@{-->}[d]^-{\Pi_\mathrm{p}} & & & & & & & & \\
 {B_2} & &
 }}
 \]
\caption{The Fano snake}\label{fig:FanoSnake}
\end{figure}
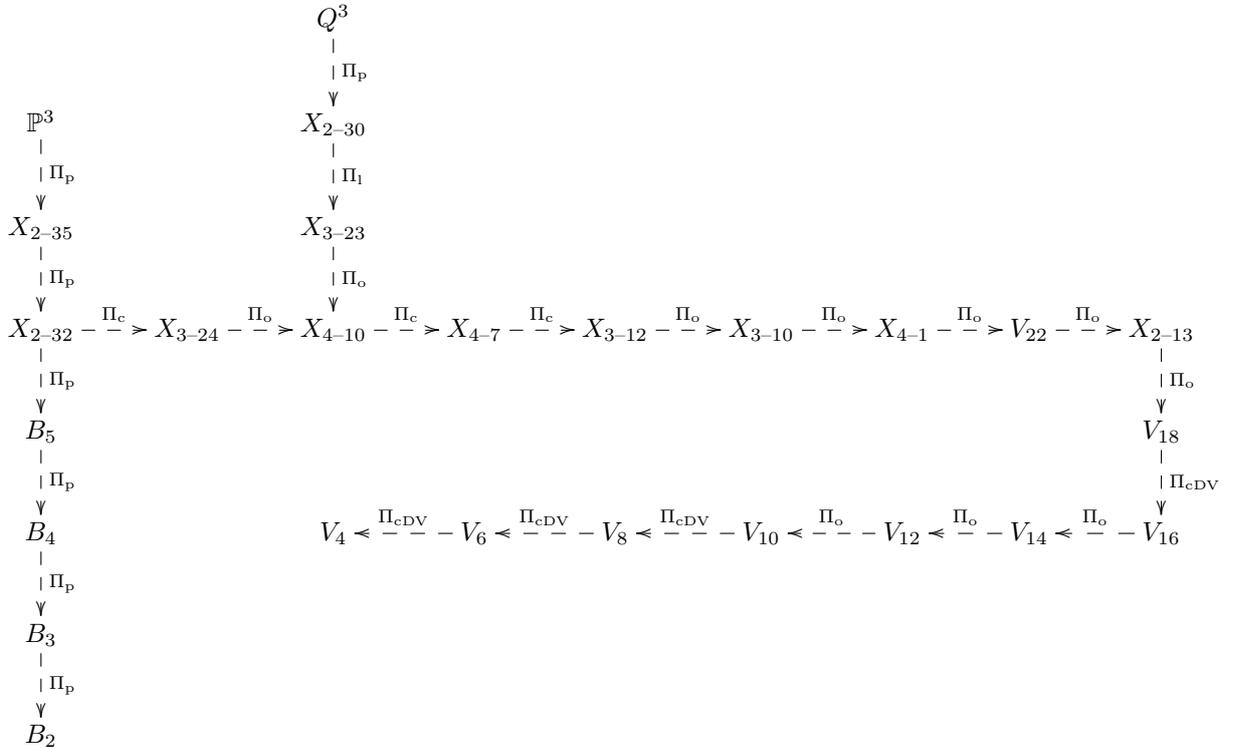

\section{Toric Landau--Ginzburg models}\label{sec:toric_LG_models}
In this section we define a toric Landau--Ginzburg model, the main object of our study. For details and examples see~\cite{Prz07,Prz09,Prz18,HV00,Bat04,Aur07,CCGGK13,CCGK13} and references therein.

Let $X$ be a smooth Fano variety of dimension $n$ and Picard rank $\rho$. Fix a basis $\{H_1,\ldots,H_\rho\}$ in $H^2(X)$ so that for any $i\in [\rho]$ and any curve $\beta$ in the K\"ahler cone
$K$ of $X$ one has $H_i\cdot\beta\ge 0$. Introduce formal variables $q_i:=q^{\tau_i}$ for each $i\in [\rho]$, and for any $\beta\in H_2(X)$ define
\[
q^\beta:=q^{\sum\tau_i(H_i\cdot\beta)}.
\]
Consider the Novikov ring $\CC_q$, i.e.~a group ring for $H_2(X)$. We treat it as a ring of polynomials over $\CC$ in formal variables $q^\beta$, with relations $q^{\beta_1}q^{\beta_2}=q^{\beta_1+\beta_2}$. Notice that for any $\beta\in K$ the monomial $q^\beta$ has non-negative degrees in the $q_i$.

Let the number
\[
\langle\tau_a\gamma\rangle_\beta,\qquad\text{ where }a\in\ZZ_{\geq 0},\gamma\in H^*(X),\beta \in K,
\]
be a one-pointed Gromov--Witten invariant with descendants for $X$; see~\cite[VI-2.1]{Ma99}. Let $\mathbf{1}$ be the fundamental class of $X$. The series
\[
I^X_{0}(q_1,\ldots,q_{\rho})=1+\sum_{\beta\in K}\langle\tau_{-K_X\cdot\beta-2}\mathbf{1}\rangle_{-K_X\cdot\beta}\cdot q^\beta
\]
is called the \emph{constant term of $I$-series} (or the \emph{constant term of Givental's $J$-series}) for $X$, and the series
\[
\widetilde{I}^X_{0}(q_1,\ldots,q_{\rho})=1+\sum_{\beta\in K}(-K_X\cdot\beta)!\langle\tau_{-K_X\cdot\beta-2}\mathbf{1}\rangle_{-K_X\cdot\beta}\cdot q^\beta
\]
is called the \emph{constant term of regularised $I$-series} for $X$. Given a divisor $H=\sum\alpha_i H_i$, one can restrict these series to a direction corresponding to the given divisor by setting $\tau_i=\alpha_i\tau$ and $t=q^\tau$. Thus one can define the restriction of the constant term of regularised $I$-series to the anticanonical direction, referred to as the \emph{regularised quantum period}. This has the form
\[
\widetilde{I}^X(t)=1+a_1t+a_2t^2+\ldots.
\]

\begin{definition}\label{definition: toric LG}
A \emph{toric Landau--Ginzburg model} is a Laurent polynomial $f\in \CC[x_1^{\pm 1},\ldots,x_n^{\pm}]$ which satisfies three conditions:
\begin{enumerate}
\item\label{item:toric_LG_period}\emph{(Period condition)}
The constant term of $f^k$ is $a_k$, for each $k\in\ZZ_{>0}$;
\item\emph{(Calabi--Yau condition)}
Any fibre of $f\colon (\Cstar)^n\rightarrow\CC$ has trivial dualising sheaf;
\item\label{item:toric_LG_toric}\emph{(Toric condition)}
There exists an embedded degeneration $X\rightsquigarrow T$ to a toric variety $T$ whose fan is equal to the spanning fan of $\Newt{f}$, the Newton polytope of $f$.
\end{enumerate}
\end{definition}

Given a Laurent polynomial $f\in\CC[x^{\pm1}_1,\ldots,x^{\pm1}_n]$, the \emph{classical period} of $f$ is given by:
\[
\pi_f(t)=\left(\frac1{2\pi i}\right)^n\int_{\abs{x_1}=\ldots=\abs{x_n}=1}\frac{1}{1-tf(x_1,\dots, x_n)}
\frac{dx_1}{x_1}\cdots\frac{dx_n}{x_n},\qquad\text{ for }t\in\CC,\ \abs{t}\ll\infty.
\]
This gives rise to a solution of a GKZ hypergeometric differential system associated to the Newton polytope of $f$. Expanding this integral by the formal variable $t$ one obtains a generating series for the constant terms of exponents of $f$, which we call the \emph{period sequence} and also denote by $\pi_f$:
\[
\pi_f(t)=1+\sum_{k\geq1}\coeff_{1}(f^k)t^k.
\]
For details see~\cite[Proposition~2.3]{Prz07} and~\cite[Theorem~3.2]{CCGGK13}. The period condition in Definition~\ref{definition: toric LG} tells us that we have an isomorphism between the Picard--Fuchs differential equation for a family of fibres of the map $f\colon(\Cstar)^n\rightarrow\CC$, and the regularised quantum differential equation for $X$.

The Calabi--Yau condition is motivated by the following (see, for example,~\cite[Principle~32]{Prz09}):

\begin{principle}[Compactification principle]\label{principle:compactification}
The relative compactification of the family of fibres of a ``good'' toric Landau--Ginzburg model (defined up to flops) satisfies the (B-side of the) Homological Mirror Symmetry conjecture.
\end{principle}

\noindent
From the point of view of Homological Mirror Symmetry, the fibres of the Landau--Ginzburg model for a Fano variety are Calabi--Yau varieties. The Calabi--Yau condition is designed to remove any obstructions for the compactification principle in this case. Finally, the toric condition is a generalisation of Batyrev's principle for small toric degenerations.

\begin{claim}
All toric Landau--Ginzburg models associated with the same toric degeneration of $X$ have the same support. In other words, given a toric degeneration $X\rightsquigarrow T$, by varying a symplectic form on $X$ one can vary the coefficients of the Laurent polynomial $f_X$ whilst keeping the Newton polytope fixed.
\end{claim}

\begin{conjecture}[Strong version of Mirror Symmetry for variation of Hodge structures]\label{conjecture:MSVHS}
Any smooth Fano variety has a toric Landau--Ginzburg model.
\end{conjecture}

\noindent
Some progress has been made towards proving Conjecture~\ref{conjecture:MSVHS} in the case of threefolds. \cite[Proposition 9 and Theorem 14]{Prz09} and \cite[Theorem 2.2 and Theorem 3.1]{ILP11} tell us the following.

\begin{theorem}
Conjecture~\ref{conjecture:MSVHS} holds for Picard rank $1$ Fano threefolds and for complete intersections.
\end{theorem}

Moreover, Period condition for all Fano threefolds holds by~\cite{CCGGK13}, and Calabi--Yau condition holds for them by~\cite{Prz17}. Thus, Theorem~\ref{theorem: toric LG's}
implies the following.

\begin{corollary}
Conjecture~\ref{conjecture:MSVHS} holds for smooth Fano threefolds.
\end{corollary}

The compactification principle requires that the fibres of a Calabi--Yau compactification of a toric Landau--Ginzburg model for $X$ are mirror dual to anticanonical sections of $X$. In the threefold case, this duality is called the \emph{Dolgachev--Nikulin--Pinkham duality}~\cite{DN77,Pin77}, and can be formulated in terms of orthogonal Picard lattices. In~\cite{DHKLP,Prz18} the uniqueness of compactified toric Landau--Ginzburg models satisfying these conditions is proved for rank one Fano threefolds, and so the theorem holds for all Dolgachev--Nikulin toric Landau--Ginzburg models. 

\section{Mutations}\label{sec:mutation}
Mutations are a special class of birational transformations which act on Laurent polynomials and arise naturally in the context of mirror symmetry for Fano manifolds. As discussed in~\S\ref{sec:toric_LG_models}, an $n$-dimensional Fano manifold $X$ is expected to correspond to a Laurent polynomial $f\in\CC[x_1^{\pm1},\ldots,x_n^{\pm1}]$, with the period sequence $\pi_f$ of $f$ agreeing with the regularised quantum period of $X$. This correspondence is far from unique: typically there will be infinitely many Laurent polynomials corresponding to a given Fano manifold, and it is expected that these Laurent polynomials are related via mutation~\cite{GU10,ACGK12,KP12,GHK13}.

We recall the definition of mutation as given in~\cite{ACGK12}. Write $f$ in the form
\[
f=\sum_{h=\hmin}^\hmax C_h(x_1,\ldots,x_{n-1})x_n^h,\qquad\text{ for some }\hmin<0\text{ and }\hmax>0,
\]
where each $C_h$ is a Laurent polynomial in $n-1$ variables, and let $F\in\CC[x_1^{\pm1},\ldots,x_{n-1}^{\pm1}]$ be a Laurent polynomial  such that $C_h$ is divisible by $F^{\abs{h}}$ for each $h\in\{\hmin,\hmin+1,\ldots,-1\}$. We call $F$ a \emph{factor}. Define a birational transformation $\varphi:(\Cstar)^n\dashrightarrow(\Cstar)^n$ by
\[
\left(x_1,\ldots,x_{n-1},x_n\right)\mapsto\left(x_1,\ldots,x_{n-1},F(x_1,\ldots,x_{n-1})x_n\right).
\]
The pullback of $f$ by $\varphi$ gives a Laurent polynomial
\[
g:=\varphi^*(f)=\sum_{h=\hmin}^\hmax F(x_1,\ldots,x_{n-1})^hC_h(x_1,\ldots,x_{n-1})x_n^h.
\]
Notice that the requirement that $F^{\abs{h}}$ divides $C_h$ for each $h\in\{\hmin,\hmin+1,\ldots,-1\}$ is essential: it is precisely this condition that ensures that $g$ a Laurent polynomial.

\begin{definition}
A \emph{mutation} (or \emph{symplectomorphism of cluster type}) is the birational transformation $\varphi$, possibly pre- and post-composed with a monomial change of basis. We say that $f$ and $g$ are \emph{related by the mutation $\varphi$}. If there exists a finite sequence of Laurent polynomials $f=f_0,f_1,\ldots,f_k=g$, where each $f_i$ and $f_{i+1}$ are related by mutation, then we call $f$ and $g$ \emph{mutation equivalent}.
\end{definition}

\noindent
One important property is that mutations preserve the classical period:
\begin{lemma}[{\!\!\cite[Lemma~1]{ACGK12}}]\label{lem:period_preserved}
If the Laurent polynomials $f$ and $g$ are mutation equivalent then $\pi_f=\pi_g$.
\end{lemma}

\begin{example}[cf.~{\cite[Example 2.6]{KP12}}]\label{eg:cubic_dim_n_mutation}
Consider the Laurent polynomial
\[
f_1:=\frac{(1+x_1+x_2)^3}{\prod_{i=1}^nx_i}+x_3+\ldots+x_n\in\CC[x_1^{\pm1},\ldots,x_n^{\pm1}],
\]
where $n\geq 3$. This is a weak Landau--Ginzburg model for the $n$-dimensional cubic as described in, for example,~\cite[\S2.1]{ILP11}. By, for example,~\cite[Corollary~D.5]{CCGK13} it has period sequence:
\[
\pi(t)=\sum_{k=0}^\infty\frac{(3k)!((n-1)k)!}{(k!)^{n+2}}t^{(n-1)k}.
\]
Set $F:=1+x_1+x_2$ and define the map
\[
\varphi_1:(x_1,x_2,x_3,x_4,\ldots,x_n)\mapsto (x_1,x_2,Fx_3,x_4,\ldots,x_n).
\]
Then $\varphi_1$ defines a mutation of $f_1$ (in this case $x_3$ is playing the role of $x_n$ in the definition above):
\[
f_2:=\varphi_1^*(f_1)=\frac{(1+x_1+x_2)^2}{\prod_{i=1}^nx_i}+(1+x_1+x_2)x_3+x_4+\ldots+x_n.
\]
If $n\geq 4$, we can define
\[
\varphi_2:(x_1,x_2,x_3,x_4,x_5,\ldots,x_n)\mapsto(x_1,x_2,x_3,Fx_4,x_5,\ldots,x_n),
\]
giving us a second mutation:
\[
f_3:=\varphi_2^*(f_2)=\frac{1+x_1+x_2}{\prod_{i=1}^n{x_i}}+(1+x_2+x_2)(x_3+x_4)+x_5+\ldots+x_n.
\]
We could attempt to continue this process. If $n\geq 5$ we define the map
\[
\varphi_3:(x_1,\ldots,x_4,x_5,x_6,\ldots,x_n)\mapsto(x_1,\ldots,x_4,Fx_5,x_6,\ldots,x_n).
\]
This gives a mutation
\[
f_4:=\varphi_3^*(f_3)=\frac{1}{\prod_{i=1}^nx_i}+(1+x_1+x_2)(x_3+x_4+x_5)+x_6+\ldots+x_n.
\]
However, $f_3\cong f_4$ via the obvious monomial change of basis, and so we regard these two weak Landau--Ginzburg models as being essentially the same.
\end{example}

Suppose that $f$ is mirror to an $n$-dimensional Fano manifold $X$, and consider the Newton polytope $P:=\Newt{f}\subset\NRR$ of $f$. We may assume without loss of generality that $P$ is of maximum dimension and that it contains the origin strictly in its interior. The \emph{spanning fan} of $P$ -- the fan whose cones in $\NRR$ are generated by the faces of $P$ -- gives rise to a toric variety $X_P$, which we call a \emph{toric model} for $X$. In general $X_P$ will be singular. However, it is expected to admit a smoothing with general fibre $X$. Since a smooth Fano manifold can degenerate to many different singular toric varieties, we expect many different mirrors for $X$. This is reflected in that fact that $f$ is mutation equivalent to many different Laurent polynomials, all of which have the same classical period by Lemma~\ref{lem:period_preserved}.

\begin{example}\label{eg:P2_mutation_tree}
The Laurent polynomial $f=x+y+\frac{1}{xy}\in\CC[x^{\pm1},y^{\pm1}]$ has Newton polytope $P=\sconv{(1,0),(0,1),(-1,-1)}\subset\NRR$ and corresponding toric variety $X_P=\PP^2$. Via mutation, we obtain $g=y+2x^2y^2+x^4y^3+\frac{1}{xy}$, giving the toric model $X_Q=\PP(1,1,4)$. Notice that the singular point $\frac{1}{4}(1,1)$ is a $T$-singularity and hence admits a $\QQ$-Gorenstein (qG) one-parameter smoothing~\cite{Wah80,KS-B88}. We can continue mutating, resulting in a directed graph of toric models of the form $X_{(a,b,c)}:=\PP(a^2,b^2,c^2)$, where $(a,b,c)\in\ZZ_{>0}^3$ is a solution to the Markov equation $3abc=a^2+b^2+c^2$~\cite{HP10,AK13}. There are infinitely many positive solutions to the Markov equation, obtainable via cluster-style mutations of the form $(a,b,c)\mapsto(b,c,3bc-a)$, and so we have infinitely many toric models $X_{(a,b,c)}$ for $\PP^2$. In each case the weighted projective space $X_{(a,b,c)}$ has only $T$-singularities and so admits a qG-smoothing.
\end{example}

A mutation between two Laurent polynomials induces a mutation of the corresponding Newton polytopes $P$ and $Q$. The corresponding toric varieties $X_P$ and $X_Q$ are deformation equivalent in the following precise sense:

\begin{lemma}[{\!\!\cite[Theorem~1.3]{Ilt12}}]\label{lem:flat_family}
Let $X_P$ and $X_Q$ be related by mutation. Then there exists a flat family $\calX\rightarrow\PP^1$ such that $\calX_0\cong X_P$ and $\calX_\infty\cong X_Q$.
\end{lemma}

Very little is known about the behaviour of mutations in general. The two-dimensional setting has been studied in~\cite{GU10,AK13,AK14,KNP14,Pragmatic,Pri15}. In particular, the toric surface $X_P$ is qG-smoothable to $X$~\cite[Theorem~3]{Pragmatic}, and, for each del~Pezzo surface $X$, the set of all polygons $P$ such that $X_P$ is qG-smoothable to $X$ forms a single mutation-equivalence class~\cite[Theorem~1.2]{KNP14}. 

In dimension three a special class of Laurent polynomials called \emph{Minkowski polynomials} were introduced in~\cite{ACGK12}. These are Laurent polynomials $f$ in three variables whose Newton polytope is a reflexive polytope, with $f$ satisfying certain additional conditions. There are $3747$ Minkowski polynomials (up to monomial change of basis), and together they generate $165$ periods. Furthermore, any two Minkowski polynomials have the same period sequence if and only if they are mutation equivalent. Of these periods, $98$ are of so-called \emph{manifold type}, with the remaining $67$ being of \emph{orbifold type} (these are properties of the associated Picard--Fuchs differential equations; see~\cite[\S7]{CCGGK13} for the definitions). In~\cite{CCGK13}, the period sequences of manifold type were shown to correspond under mirror symmetry to the $98$ deformation families of three-dimensional Fano manifolds $X$ with very ample anticanonical bundle $-K_X$.

\begin{proposition}
Let $P\subset\NRR$ be a reflexive polytope, with associated Gorenstein toric Fano variety $X_P$. Assume that $X_P$ has a smooth deformation space, and is a degeneration of a generic smooth Fano variety $X$. Let $Q$ be any mutation of $P$, with associated toric variety $X_Q$. Then $X_Q$ is also a degeneration of $X$.
\end{proposition}

\begin{proof}
By Lemma~\ref{lem:flat_family} there is a flat projective family over $\PP^1$ with $X_P$ and $X_Q$ as special fibres. Since $[X_P]\in\calH_X$ lies on a single irreducible component, $[X_Q]$ must also lie on this component. A general point of this component is a smooth Fano threefold deformation equivalent to $X$, so $X$ degenerates to $X_Q$.
\end{proof}

By applying the above proposition, we may find many more toric degenerations of a given Fano threefold. Indeed, suppose that we have a smooth Fano threefold $X$ which degenerates to a Gorenstein toric Fano variety $X_P$ having a smooth deformation space, as described in~\S\S\ref{app:small},~\ref{app:lowdeg},~\ref{app:tri},~\ref{app:prod}, and~\ref{app:ci}. We can use mutation to construct other Gorenstein toric Fano varieties to which $X$ degenerates. By consulting the calculations of~\S\S\ref{app:compdims} and~\ref{app:tsdims} 
we can determine which of these have smooth deformation spaces, and then iterate using these new examples. We record the resulting degenerations in~\S\ref{app:mutation}.

\section{Toric degenerations}\label{sec:degen_intro}
Let $X$ be a smooth Fano threefold with a very ample anticanonical divisor, and consider its anticanonical embedding $V\hookrightarrow\PP^n$. As explained in~\S\ref{sec:threefold_case}, we are interested in finding embedded degenerations of $X$ to Gorenstein toric Fano varieties. The main result of this section is the following:

\begin{theorem}\label{theorem: toric LG's}
Let $X$ be a generic smooth Fano threefold with a very ample anticanonical divisor. Then $X$ has an embedded degeneration to a Gorenstein toric Fano variety $X'$ with smooth deformation space.
\end{theorem}

Our proof makes use of the classification of smooth Fano threefolds~\cite{Is77,Is78,MM81,MM03} and consists of several steps:

\begin{enumerate}
\item\label{stepone}
We use several techniques, described in~\S\S\ref{sec:known_degenerations}--\ref{sec:exceptional_case} below, to construct at least one degeneration $X'$ for each smooth Fano $X$. The vast majority of these come from the models constructed in~\cite{CCGK13} of smooth Fano threefolds as complete intersections in toric varieties; see~\S\ref{sec:ci}.
Note that after we constructed these degenerations, the papers~\cite{DH16,CKP,DHKLP} appeared, which provide further tools for systematically constructing toric degenerations.
\item\label{steptwo}
For any Fano variety $X$ with very ample anticanonical divisor, let $\calH_X$ denote the Hilbert scheme parameterizing projective schemes with the same Hilbert polynomial as $X$ in its anticanonical embedding. If $X$ is smooth, then it corresponds to a smooth point on an irreducible component of $\calH_X$. For each smooth Fano $X$, we compute the dimension of this component using~\cite[Proposition~3.1]{CI14}.
\item\label{stepthree}
Using the comparison theorem of Kleppe~\cite[Theorem~3.6]{Klep79} we compute the tangent space dimensions for all Gorenstein toric Fano threefolds, viewed as points in relevant Hilbert schemes.
\item\label{stepfour}
Using steps~\eqref{steptwo} and~\eqref{stepthree} we check that all special fibres $X'$ of the degenerations in step~\eqref{stepone} have tangent space dimension equal to the dimension of the Hilbert scheme component of $X$, and hence correspond to smooth points in the Hilbert scheme. Since the forgetful functor to the deformation space of $X'$ is smooth~\cite[\S2.1]{CI14}, each such $X'$ has smooth deformation space.
\end{enumerate}

\noindent
None of our results claim to be exhaustive; for example, it is certainly possible for there to exist degenerations which do not appear in our search at step~\eqref{stepone}.

\subsection{Previously known degenerations}\label{sec:known_degenerations}
We make use of a number of known toric degenerations. Galkin~\cite{Gal07} classified all degenerations of smooth Fano threefolds to Gorenstein toric Fano varieties with at worst terminal singularities, which are called \emph{small} toric degenerations. These are recorded in~\S\ref{app:small}. Notice that these include the $18$ smooth toric Fano threefolds classified by Batyrev~\cite{Bat81} and Watanabe--Watanabe~\cite{WW82}. The embedded degenerations of smooth Fano threefolds with a very ample anticanonical divisor to Gorenstein toric Fano varieties of degree at most twelve were classified in~\cite{CI14}. These are recorded in~\S\ref{app:lowdeg}.
When the paper was written, the preprint~\cite{Pr19}, where using different methods Gorenstein toric degenerations of Fano threefolds were studied, appeared.

\subsection{Triangulations of the moment polytope}\label{sec:triang}
We use the techniques developed in~\cite{CI12} to construct degenerations of rank one, index one Fano threefolds. Consider the two-dimensional simplicial complexes $T_i$, $i\in\{4,5,6,7,8,9,10,11\}$, defined as follows: when $i=4$ we  define $T_4$ to be the boundary of the three-simplex; when $i=5$ we define $T_5$ to be the bipyramid over the boundary complex of a two-simplex; when $6\leq i\leq 10$, let $T_i$ be the unique triangulation of the sphere with $i$ vertices having valencies four and five; when $i=11$, let $T_{11}$ be the unique triangulation of the sphere having valencies $4,4,5,5,5,5,5,5,5,5,6$.

\begin{theorem}[{\!\!\cite[\S3]{CI12}}]\label{thm:nice}
For $4\leq i\leq 11$, let $d:=2i-4$ and let $X_d$ be a general rank one, index one, degree $d$ smooth Fano threefold. Consider a three-dimensional reflexive polytope $\Delta\subset\MRR$ whose boundary admits a regular unimodular triangulation of the form $T_i$, and let $X':=X(\Delta)$ be the associated Gorenstein toric Fano variety with fan normal to $\Delta$. If $i<11$, or if $i=11$ and $h^0(X',\calN_{X'/\PP^i})=153$, then $X_d$ has an embedded degeneration to $X'$. Furthermore, the point corresponding to $X'$ in $\calH_{X_d}$ is smooth.
\end{theorem}

\noindent
The resulting degenerations are recorded in~\S\ref{app:tri}.

\subsection{Products with del~Pezzo surfaces}
A number of Fano threefolds are products of del~Pezzo surfaces with $\PP^1$:
$\XX{5}{3}, \XX{6}{1}, \XX{7}{1}, \XX{8}{1}$, and $\XX{9}{1}$, of which the first three have very ample anticanonical divisor. Toric degenerations of these threefolds may be found by degenerating the corresponding del~Pezzo surfaces. In many cases these degenerations are well understood; see for example the work of Hacking--Prokhorov~\cite{HP10}. The threefold degenerations constructed this way are recorded in~\S\ref{app:prod}.

\begin{figure}[ht]
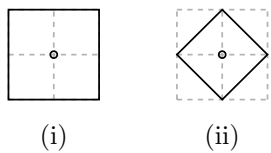

\centering
\begin{tabular}{ccc}
\includegraphics[scale=0.65]{DP4_13}&\hspace{3em}&
\includegraphics[scale=0.65]{DP8_2}\\
\hypertarget{fig:dp1}{(i)}&&\hypertarget{fig:dp2}{(ii)}
\end{tabular}
\caption{The spanning polytope~(i) $P\subset\NRR$ and the moment polytope~(ii) $\dual{P}\subset\MRR$ of a toric degeneration of the del~Pezzo surface of degree four.}\label{fig:dp}
\end{figure}

\begin{example}
Consider the del~Pezzo surface $S_4$ of degree four given by the blow-up of $\PP^2$ in five points. This degenerates to the toric variety whose spanning polytope is picture in Figure~\ref{fig:dp}(\hyperlink{fig:dp1}{i}). The Fano threefold $\XX{7}{1}$ is $S_4\times\PP^1$, and thus degenerates to the toric Fano threefold $X_P$ where, up to $\GL_3(\ZZ)$-action, $P:=\sconv{\pm(1,1,0),\pm(1,-1,0),\pm(0,0,1)}\subset\NRR$. This has Reflexive ID $510$.
\end{example}

\subsection{Complete Intersections in toric varieties}\label{sec:ci}
It appears to be common for smooth Fano varieties to arise as complete intersections in the homogeneous coordinate ring of a smooth toric Fano variety. Eight of the del~Pezzo surfaces can be realised this way; so can at least $78$ of the $105$ smooth Fano threefolds~\cite{CCGK13}, and at least $738$ of the smooth Fano fourfolds~\cite{CKP14}. This leads to the following natural construction of degenerations.

Let $W$ be a smooth complete toric variety, with $I$ the set of invariant prime divisors. Its homogeneous coordinate ring $R=\CC[x_i\mid i\in I]$ is a polynomial ring graded by $\Pic{W}$. Then $W$ is the quotient of $U=\spec R\setminus Z$ by the Picard torus $T=\spec \CC[\Pic{W}]$, where $Z$ is the so-called irrelevant, or exceptional, set. Suppose that $X$ is a complete intersection in $W$ of Cartier divisors $D_1,\ldots,D_k$; let $\overline D_i$ be the class of $D_i$ in $\Pic{W}$. Each divisor $D_i$ may be encoded by a homogeneous degree $\overline D_i$ polynomial $f_i$ in $R$. Then $X$ arises as the quotient
\[
X=(U\cap V(f_1,\ldots,f_k))\GIT T.
\]
In order to degenerate $X$, we may degenerate the polynomials $f_i$ to polynomials $g_i$ of the same multidegree. As long as the degenerate polynomials $g_i$ still form a regular sequence in $R$, we get a degeneration from $X$ to $X'=(U\cap V(g_1,\ldots,g_k))\GIT T$. If $V(g_1,\ldots,g_k)\cap U\hookrightarrow U$ is an equivariant embedding of toric varieties, then the resulting quotient $X'$ is toric as well.

To construct a toric $X'$, we may choose the $g_i$ to be binomials. If $M$ is the character lattice of the torus of $W$, we have an exact sequence
\[
\begin{CD}
0@>>> M@>>>\ZZ^I@>\pi>>\Pic{W}@>>>0,
\end{CD}
\]
and each binomial $g_i$ determines a rank-one sublattice $L_i$ of $M$. Let $O$ be the positive orthant of $\ZZ^I\otimes\QQ$, and let $O'$ be its image in $\left(\ZZ^I/\sum L_i\right)\otimes\QQ$.
If $M':=\ZZ^I/\sum L_i$ is torsion free and the natural map $R\rightarrow\CC[O'\cap M']$ is surjective with kernel generated by the $g_i$, then $V(g_1,\ldots,g_k)$ is toric, with $U\cap V(g_1,\ldots,g_k)\hookrightarrow U$ an equivariant embedding. In order to determine the quotient in this situation, let $A$ be any ample class in $\Pic{W}$. Then the quotient $X'$ of $U\cap V(g_1,\ldots,g_k)$ is the toric variety whose moment polytope $\Delta$ is the image of $\pi^{-1}(A)\cap O$ in $M'\otimes\QQ$.

\begin{example}
The smooth Fano threefold $X=\XX{2}{7}$ is a codimension two complete intersection in the toric variety $W=\PP^1\times \PP^4$. The homogeneous coordinate ring $R=\CC[x_0,x_1,y_0,\ldots,y_4]$ of $W$ has a $\Pic{W}=\ZZ^2$-grading given by the columns of the weight data
\[
\begin{array}{cccccccl}
x_0&x_1&y_0&y_1&y_2&y_3&y_4&\\
\cmidrule{1-7}
1&1&0&0&0&0&0&\hspace{1.5ex}A\\
0&0&1&1&1&1&1&\hspace{1.5ex}B
 \end{array}
\]
with ample cone $\overline{\Amp}\,W=\left<A,B\right>$. The threefold $X$ is the intersection of divisors $(A+2B)\cap(2B)$. Consider the $\ZZ^2$-homogeneous polynomials
\begin{align*}
g_1&=x_0y_0y_1-x_1y_2y_3\\
g_2&=y_3^2-y_2y_4.
\end{align*}
Then $V(g_1,g_2)$ is toric, and the resulting quotient by the Picard torus gives the Gorenstein toric Fano variety $X'$ with Reflexive ID $3813$ given by the moment polytope
\[
\Delta=\mathrm{conv}\{(1,0,0),(1,1,0),(1,0,1),(1,1,1),(0,0,1),(0,-1,1),(-1,0,-1),(-1,1,-1)\}.
\]
\end{example}

In the above construction, the resulting toric variety $X'$ need not be Fano. However, in the case when $X'$ is a weak Fano with Gorenstein singularities we can construct a degeneration from $X$ to the anticanonical model $X''$ of $X'$. Indeed, let $\pi:\calX\to S$ be the total space of the degeneration of $X$ to $X'$. Since all fibres are Cohen-Macaulay, ${-K_{\calX/S}}{|_F}\cong -K_F$ for any fibre $F$~\cite[Theorem~3.5.1]{Con00}. Furthermore, since $-K_{X'}$ is Cartier and nef (and $X'$ is toric), all higher cohomology of $-K_F$ vanishes for every fibre $F$. By cohomology and base change, the restriction map $H^0(\calX,-K_{\calX/S})\to H^0(F,-K_F)$ is surjective. Hence, taking $\proj$ of the section ring of $-K_{\calX/S}$ gives a flat family over $S$ with $X''$ as the special fibre and $X$ as the general fibre.

We apply the above discussion to find toric degenerations of many of the smooth Fano threefolds, using their descriptions in~\cite{CCGK13} as complete intersections in toric varieties. These degenerations are recorded in~\S\ref{app:ci}. For each degeneration, we also record the corresponding regular sequence $g_i$, using the same ordering on the homogeneous coordinates as in~\cite{CCGK13}.
For a similar approach to constructing degenerations, see also \cite{harder}.

\subsection{The exceptional case: $\XX{2}{14}$}\label{sec:exceptional_case}
The previous methods fail to construct toric degenerations for the smooth Fano threefold $\XX{2}{14}$. We now discuss a modification of~\S\ref{sec:ci} which does produce a toric degeneration.

The smooth Fano threefold $X=\XX{2}{14}$ may be realised as a divisor of bidegree $(1,1)$ in $V_5\times\PP^1$, where $B_5$ is a codimension-three linear section of the Grassmannian $\Gr(2,5)$ in its Pl\"ucker embedding~\cite{CCGK13}. The anticanonical embedding places $V_5$ in $\PP^6$ as the intersection of five quadrics. The approach of~\S\ref{sec:triang} can be applied to find degenerations of $V_5$. In particular, it degenerates to the Gorenstein toric Fano with Reflexive ID $68$.

We can realize $X$ as the intersection in the toric variety $\PP^6\times\PP^1$ of $V_5\times\PP^1$ and a divisor of bidegree $(1,1)$, defined by a bihomogeneous polynomial $f$. Simultaneously degenerating the quadrics of $V_5$ and the polynomial $f$ leads to a degeneration of $X$. More precisely, we may degenerate the quadrics to
\begin{align*}
x_1x_6-x_2x_5,\qquad x_1x_6-x_3x_4,\qquad x_0x_3-x_1x_2,\\
x_0x_5-x_1x_4,\qquad x_0x_6-x_2x_4,
\end{align*}
and the polynomial $f$ to $g=x_3x_7-x_4x_8$.
The resulting variety is the Gorenstein toric Fano threefold with Reflexive ID $2353$, generated by the polytope
\begin{align*}
P:=\mathrm{conv}\{(1,0,0),(0,1,0),&\pm(0,0,1),\pm(1,-1,0),\\
\pm(1,0,-1),&\pm(0,1,-1),(-1,-1,0),(-1,-1,1)\}\subset\NRR.
\end{align*}


\section{The graph of reflexive polytopes}\label{sec:minimality}
In~\cite{Sat00} Sato investigated the concept of $F$-equivalence classes of a smooth Fano polytope. Recall that a smooth Fano polytope $P\subset\NRR$ corresponds to a smooth toric Fano variety $X_P$ via its spanning fan. Smooth Fano polytopes are necessarily reflexive, and a reflexive polytope $P\subset\NRR$ is smooth if, for each facet $F$ of $P$, the vertices $\V{F}$ give a $\ZZ$-basis for the underlying lattice $N$ (and hence, in particular, $P$ needs be simplicial).

\begin{definition}[$F$-equivalence of smooth Fano polytopes]\label{defn:F_equivalence}
Two smooth Fano polytopes $P,Q\in\NRR$ are \emph{$F$-equivalent}, and we write $P\Fequiv Q$, if there exists a finite sequence $P_0,P_1,\ldots,P_k\subset\NRR$ of smooth Fano polytopes satisfying:
\begin{enumerate}
\item\label{item:F_equivalence_smooth_1}
$P$ and $Q$ are $\GL_n(\ZZ)$-equivalent to $P_0$ and $P_k$, respectively;
\item\label{item:F_equivalence_smooth_2}
for each $1\leq i\leq k$ we have either that $\V{P_i}=\V{P_{i-1}}\cup\{w\}$, where $w\not\in\V{P_{i-1}}$, or that $\V{P_{i-1}}=\V{P_i}\cup\{w\}$, where $w\not\in\V{P_i}$;
\item\label{item:F_equivalence_subdivision}
if $w\in\V{P_i}\setminus\V{P_{i-1}}$ then there exists a proper face $F$ of $P_{i-1}$ such that
\[
w=\sum_{v\in\V{F}}v
\]
and the set of facets of $P_i$ containing $w$ is equal to
\[
\left\{\conv{\{w\}\cup\V{F'}\setminus\{v\}}\mid F'\text{ is a facet of }P_{i_1}, F\subset F', v\in\V{F}\right\}.
\]
In other words, $P_i$ is obtained by taking a stellar subdivision of $P_{i-1}$ with $w$. Similarly, if $w\in\V{P_{i-1}}\setminus\V{P_i}$ then $P_{i-1}$ is given by a stellar subdivision of $P_i$ with $w$.
\end{enumerate}
\end{definition}

Notice that condition~\eqref{item:F_equivalence_subdivision} means that if $P\Fequiv Q$ then the corresponding smooth toric Fano varieties $X_P$ and $X_Q$ are related via a sequence of equivariant blow-ups or blow-downs. Little is known about $F$-equivalence in general, although it is known that all smooth Fano polytopes are $F$-equivalent in dimensions $\leq 4$, and that there exists non-$F$-equivalence polytopes in all dimensions $\geq 5$.

We wish to generalise the notion of $F$-equivalence to encompass projections between reflexive polytopes. Our focus here is on three-dimensions, although what follows can readily be generalised to higher dimensions. Fix a three-dimensional reflexive polytope $P\subset\NRR$ and consider the corresponding Gorenstein toric Fano variety $X=X_P$. Toric points and curves on $X$ correspond to, respectively, facets and edges of $P$. Projections from these points (or curves) correspond to blowing-up the cone generated by the relevant facet (or edge). Since we are restricting ourselves to the class of reflexive polytopes, it is clear that the original facet (or edge) only needs to be considered if it has no (relative) interior points: any interior point of that face will become an interior point of the resulting polytope, preventing it from being reflexive.

\begin{proposition}\label{prop:reflexive_facet}
Let $P\subset\NRR$ be a three-dimensional reflexive polytope, and let $F$ be a facet or edge of $P$ such that $F$ contains no (relative) interior points. Up to $\GL_3(\ZZ)$-equivalence, $F$ is one of the four possibilities shown in Table~\ref{table:proj:faceTypes}.
\end{proposition}
\begin{proof}
First consider the case when $F$ is a facet of $P$. Since $P$ is reflexive, there exists some primitive lattice point $u_F\in M$ such that $u_F(v)=1$ for each $v\in F$. In particular, by applying any change of basis sending $u_F$ to $e_1^*$ we can insist that $F$ is contained in the two-dimensional affine subspace $\Gamma:=\{(1,a,b)\mid a,b\in\QQ\}$. Pick an edge $E_1$ of $F$ such that $\abs{E_1\cap N}$ is as large as possible. Let $v\in\V{E_1}$ be a vertex of $E_1$ (and hence of $F$), and let $E_2\ne E_1$ be the second edge of $F$ such that $E_1\cap E_2=\{v\}$. Let $v_i\in E_i\cap N$ be such that $v_i-v$ is primitive. Since $\intr{F}=\emptyset$ we have that $\Delta:=\sconv{v,v_1,v_2}$ is a lattice triangle with $\abs{\Delta\cap N}=3$; that is, $\Delta$ is an empty triangle. Up to $\GL_2(\ZZ)$-equivalence, the empty triangle is unique. Hence we can apply a change of basis to the affine lattice $\Gamma\cap N$ such that $v=(1,0,0)$, $v_1=(1,1,0)$, and $v_2=(1,0,1)$. By considering the possible lengths of the edges $E_1$ and $E_2$, and remembering that $(1,1,1)$ cannot be an interior point, we find the first three cases in Table~\ref{table:proj:faceTypes}. In the case when $F$ is an edge, it must have length one and hence (up to $\GL_3(\ZZ)$-equivalence) $F=\sconv{(1,0,0),(1,1,0)}$, the fourth case in Table~\ref{table:proj:faceTypes}.
\end{proof}

\begin{corollary}\label{cor:reflexive_projections}
Let $P,Q\subset\NRR$ be two three-dimensional reflexive polytopes such that $X_Q$ is obtained from $X_P$ via a projection. Then the corresponding blow-up of the face $F$ of $P$ introduces new vertices as given in Table~\ref{table:proj:faceTypes}.
\end{corollary}
\begin{proof}
We prove this only in case~\hyperlink{fig:proj:faceTypes_2}{$2$} in Table~\ref{table:proj:faceTypes}. The remaining cases are similar. We refer to Dais' survey article~\cite{Dai02} for background, and for the combinatorial interpretation. Let
\[
C=\scone{(1,0,0),(1,a+b,0),(1,a,1),(1,0,1)}\subset\NRR,
\]
where $a,b\in\ZZ_{>0}$. This is defined by the intersection of four half-spaces of the form $\{v\in\NRR\mid u_i(v)\geq 0\}$, where the $u_i$ are given by $(0,0,1),(0,1,0),(1,0,-1),(a+b,-1,-b)\in M$.  Moving these half-spaces in by one, we obtain the polyhedron:
\[
\bigcap_{i=1}^4\{v\in\NRR\mid u_i(v)\geq 1\}=\sconv{(2,1,1),(2,2a+b-1,1)}+C.
\]
Hence the blow-up is given by the subdivision of $C$ into four cones generated by inserting the rays $(2,1,1)$ and $(2,2a+b-1,1)$.
\end{proof}

\begin{table}[tbp]
\centering
 \begin{tabular}{ccl}
 \toprule
  &Face&Points added\\
\midrule
\oddrow 1\hypertarget{fig:proj:faceTypes_1}{}&\begin{tikzpicture}[baseline=(current bounding box.west)]
     \filldraw [black] (0,0) circle (2pt)
                       (0,1) circle (2pt)
                       (1,0) circle (2pt);
     \draw (0,0) -- (0,1)
           (0,0) -- (1,0)
           (1,0) -- (0,1);
     \draw (0,0) node[anchor=north east] {$(1,0,0)$};
     \draw (0,1) node[anchor=south east] {$(1,0,1)$};
     \draw (1,0) node[anchor=north west] {$(1,1,0)$};
    \end{tikzpicture}
   &$(3,1,1)$\Tstrut\\
\evnrow   2\hypertarget{fig:proj:faceTypes_2}{}&\begin{tikzpicture}[baseline=(current bounding box.west)]
     \filldraw [black] (0,0) circle (2pt)
                       (0,1) circle (2pt)
                       (1.5,0) circle (2pt)
                       (3,0) circle (2pt)
                       (1.5,1) circle (2pt);
     \draw (0,0) -- (0,1)
           (0,0) -- (0.5,0)
           (1,0) -- (2,0)
           (2.5,0) -- (3,0)
           (3,0) -- (1.5,1)
           (0,1) -- (0.5,1)
           (1,1) -- (1.5,1);
     \draw [dotted] (0.5,1) -- (1,1)
                    (0.5,0) -- (2.5,0);
     \draw (0,0) node[anchor=north east] {$(1,0,0)$};
     \draw (0,1) node[anchor=south east] {$(1,0,1)$};
     \draw (1.5,0) node[anchor=north] {$(1,a,0)$};
     \draw (3,0) node[anchor=north west] {$(1,a+b,0)$};
     \draw (1.5,1) node[anchor=south] {$(1,a,1)$};
    \end{tikzpicture}
   &$(2,1,1)$, $(2,2a+b-1,1)$\Tstrut\\
\oddrow  3&\begin{tikzpicture}[baseline=(current bounding box.west)]
     \filldraw [black] (0,0) circle (2pt)
                       (0,1) circle (2pt)
                       (0,2) circle (2pt)
                       (1,0) circle (2pt)
                       (2,0) circle (2pt)
                       (1,1) circle (2pt);
     \draw (0,0) -- (0,2)
           (0,0) -- (2,0)
           (2,0) -- (0,2);
     \draw (0,0) node[anchor=north east] {$(1,0,0)$};
     \draw (0,1) node[anchor=east] {$(1,0,1)$};
     \draw (0,2) node[anchor=south east] {$(1,0,2)$};
     \draw (1,0) node[anchor=north] {$(1,1,0)$};
     \draw (2,0) node[anchor=north west] {$(1,2,0)$};
     \draw (1,1) node[anchor=south west] {$(1,1,1)$};
    \end{tikzpicture}&$(3,2,2)$\Bstrut\\
\evnrow  \raisebox{-2px}{4}&\begin{tikzpicture}[baseline=(current bounding box.west)]
     \filldraw [black] (0,0) circle (2pt)
                       (1,0) circle (2pt);
     \draw (0,0) -- (1,0);
     \draw (0,0) node[anchor=east] {$(1,0,0)$};
     \draw (1,0) node[anchor=west] {$(1,1,0)$};
    \end{tikzpicture}&\raisebox{-2px}{$(2,1,0)$}\Tstrut\\
\bottomrule
 \end{tabular}\vspace{1em}
\caption{The possible choices of faces to be blown-up, and corresponding points to be added, when considering projections between Gorenstein toric Fano threefolds.}\label{table:proj:faceTypes}
\end{table}

Of course case~\hyperlink{fig:proj:faceTypes_1}{$1$} in Table~\ref{table:proj:faceTypes} is a specialisations of case~\hyperlink{fig:proj:faceTypes_2}{$2$}. However, since the point added is what is important, we list it separately. When $a=1$, $b=0$, or when $a=0$, $b=2$ in case~\hyperlink{fig:proj:faceTypes_2}{$2$}, the coordinates of the two points to be added coincide, so we add the single point $(2,1,1)$.

Since the choice of the basic links we are using relies on distinguishing curves of small degrees on a given toric variety, it is necessary to be able to easily calculate the anticanonical degree of a given curve. This can be done as follows:

\begin{lemma}\label{lemma:curve_degree}
Let $T$ be a three-dimensional $\QQ$-factorial projective toric variety with simplicial fan $\Delta$. Let $c_1$ and $c_2$ be rays in $\Delta^{(1)}$ generating a two-dimensional cone in $\Delta^{(2)}$, with corresponding torus-invariant curve $C$. Here $\Delta^{(n)}$ denotes the set of $n$-dimensional cones in the fan $\Delta$.  Let $a_1$ and $a_2$ be rays such that $F_i=\scone{c_1,c_2,a_i}\in\Delta^{(3)}$, $F_1\neq F_2$, and let $e_1,\ldots, e_r$ denote the remaining rays of $\Delta$, so that $\Delta^{(1)}=\{c_1,c_2,a_1,a_2,e_1,\ldots,e_r\}$. Let $\Vol{F_i}=\abs{\Gamma:N}$ be the index of the sublattice $\Gamma_i$ in $N$ generated by $c_1,c_2,a_i$, where by a standard abuse of notation we confuse a ray with its primitive lattice generator in $N$; equivalently, $\Vol{F_i}$ is equal to the lattice-normalised volume of the tetrahedron $\sconv{\orig,c_1,c_2,a_i}$.

Denote the boundary divisor corresponding to $c_i$, $a_i$, or $e_i$ by $C_i$, $A_i$, or $E_i$ respectively. Let $L_1$ and $L_2$ be linear forms such that $L_1$ vanishes on $c_2$ but not on $c_1$, and $L_2$ vanishes on $c_1$ but not on $c_2$. Then the anticanonical degree of $C$ is equal to
\[
 \deg{C}=\left(1-\frac{L_1(a_1)}{L_1(c_1)}-\frac{L_2(a_1)}{L_2(c_2)}\right)\frac{1}{\Vol{F_1}}+\left(1-\frac{L_1(a_2)}{L_1(c_1)}-\frac{L_2(a_2)}{L_2(c_2)}\right)\frac{1}{\Vol{F_2}}.
\]
\end{lemma}
\begin{proof}
Recall that the anticanonical degree of $C$ is the sum of its intersection with all boundary divisors
\[
\deg{C}=C\cdot\left(C_1+C-2+A_1+A_2+\sum_{i=1}^r E_i\right)=C\cdot C_1+C\cdot C_2+C\cdot A_1+C\cdot A_2,
\]
with intersection numbers given by
\[
C\cdot A_1=1/\Vol{F_1},\qquad C\cdot A_2=1/\Vol{F_2}.
\]
Given the linear forms $L_1$, $L_2$ as above, let $D_i$ be the principal divisor corresponding to the form $L_i$. Then:
\[
 0\equiv D_i\equiv L_i(c_1)C_1+L_i(c_2)C_2+L_i(a_1)A_1+L_i(a_2)A_2+\sum_{j=1}^r{L_i(e_j)E_j}.
\]
This gives:
\[
 0=C\cdot D_i = L_i(c_i)C\cdot C_i +L_i(a_1)C\cdot A_1 + L_i(a_2)C\cdot A_2.
\]
Hence:
\begin{equation*}
 C\cdot C_i=-\frac{L_i(a_1)}{L_i(c_i)}\frac{1}{\mathrm{Vol}\left(F_1\right)}-\frac{L_i(a_2)}{L_i(c_i)}\frac{1}{\mathrm{Vol}\left(F_2\right)}.
  \qedhere
  \end{equation*}
\end{proof}

\begin{corollary}\label{cor:curve_degree:calculation}
 Let $X$ be a toric Fano threefold and let $P$ the corresponding three-dimensional Fano polytope. Let $E$ be an edge of $P$ corresponding to a curve $C$ on $X$, and let $F_1$ and $F_2$  be the two facets of $P$ meeting at $E$. Let $c_1$ and $c_2$ be the two vertices of $P$ lying on $E$, and let $a_1$ and $a_2$ be vertices of $P$ lying on $F_1\setminus E$ and $F_2\setminus E$, respectively. Then:
 \[
  \deg{C}=\frac{1}{\abs{a_1\cdot\left(c_1\times c_2\right)}}+\left(1+
  \frac{\left(c_1-c_2\right)\cdot\left(a_1\times a_2\right)}{a_1\cdot\left(c_1\times c_2\right)}
  \right)\frac{1}{\abs{a_2\cdot\left(c_1\times c_2\right)}}.
 \]
\end{corollary}

\begin{proof}
 If the point $p_1\in X$ corresponding to the face $F_1$ is singular, then one can take a small resolution $X'$ of $X$ at this point and calculate the degree of $C$ via that resolution. This is done by choosing a triangulation of $F_1$, with the result being independent of the choice. So, by picking a triangulation containing the triangle $\left(c_1,c_2,a_1\right)$, one can assume that $p_1$ is a smooth point of $X$.

 Similarly, the point $p_2\in X$ corresponding to the face $F_2$ of $P$ can also be assumed to be smooth. Therefore, $X$ satisfies the conditions of Lemma~\ref{lemma:curve_degree}.

 Take:
 \[
  L_1(x)=x\cdot\left(c_2\times a_1\right),\ L_2(x)=x\cdot\left(a_1\times c_1\right).
 \]
 Since $\mathrm{Vol}\left(F_i\right)=\abs{a_i\cdot\left(c_1\times c_2\right)}$, have:
 \[
  \deg{C}=\frac{1}{\abs{a_1\cdot\left(c_1\times c_2\right)}}+\left(1-\frac{a_2\cdot\left(c_2\times a_1\right)}{c_1\cdot\left(c_2\times a_1\right)}-\frac{a_2\cdot\left(a_1\times c_1\right)}{c_2\cdot\left(a_1\times c_1\right)}\right)\frac{1}{\abs{a_2\cdot\left(c_1\times c_2\right)}},
 \]
 which simplifies to the form above.
\end{proof}

\begin{definition}[$F$-equivalence of reflexive polytopes]\label{defn:F_equivalence_reflexive}
Two three-dimensional reflexive polytopes $P,Q\in\NRR$ are $F$-equivalent, and we write $P\Fequiv Q$, if there exists a finite sequence $P_0,P_1,\ldots,P_k\subset\NRR$ of reflexive polytopes satisfying:
\begin{enumerate}
\item
$P$ and $Q$ are $\GL_3(\ZZ)$-equivalent to $P_0$ and $P_k$, respectively;
\item
For each $1\leq i\leq k$ we have that either $\V{P_i}\subsetneq\V{P_{i-1}}$ or $\V{P_{i-1}}\subsetneq\V{P_i}$.
\item\label{item:F_equivalence_reflexive_3}
If $\V{P_i}\subsetneq\V{P_{i-1}}$ then there exists a face $F$ of $P_i$ and $\varphi\in\GL_3(\ZZ)$ such that $\varphi(F)$ is one of the seven faces in Table~\ref{table:proj:faceTypes}. Furthermore, the points $\varphi(\V{P_{i-1}}\setminus\V{P_i})$ are equal to the corresponding points in Table~\ref{table:proj:faceTypes}, and $\bdry{F}\subset\bdry{P_{i-1}}$. The case when $\V{P_{i-1}}\subsetneq\V{P_i}$ is similar, but with the roles of $P_{i-1}$ and $P_i$ exchanged.
\end{enumerate}
\end{definition}

If $P\Fequiv Q$ then the corresponding Gorenstein toric Fano threefolds $X_P$ and $X_Q$ are related via a sequence of projections.

\begin{remark}\label{rem:F_equivalence_reflexive_cases}
The requirement that $\bdry{F}\subset\bdry{P_{i-1}}$ in Definition~\ref{defn:F_equivalence_reflexive}\eqref{item:F_equivalence_reflexive_3} perhaps needs a little explanation. Consider the case when $F$ is a facet. Adding the new vertices can affect a facet $F'$ adjacent to $F$, with common edge $E$, in one of three ways. Let $u_{F'}\in M$ be the primitive dual lattice vector defining the hyperplane at height one containing $F'$. Let $v_1,\ldots,v_s\in N$ be the points to be added according to Table~\ref{table:proj:faceTypes}, so that $P_{i-1}=\conv{P_i\cup\{v_1,\ldots,v_s\}}$.
\begin{enumerate}
\item
If $u_{F'}(v_i)<1$ for $1\leq i\leq s$ then $F'$ is unchanged by the addition of the new vertices, and hence $F'$ is also a facet of $P_{i-1}$.
\item
Suppose that $u_{F'}(v_i)=1$ for $1\leq i\leq m$, and $u_{F'}(v_i)<1$ for $m+1\leq i\leq s$, for some $1\leq m\leq s$. Then the facet $F'$ is transformed to the facet $F'':=\conv{F'\cup\{v_i\mid 1\leq i\leq m}$ in $P_{i-1}$. Notice that $F'\subset F''$, but that $E$ is no-longer an edge of $F''$. This is equivalent to a blowup, followed by a contraction of the curve corresponding to $E$.
\item
The final possibility is that there exists one (or more) of the $v_i$ such that $u_{F'}(v_i)> 1$. When we pass to $P_{i-1}$ we see that $F'$ is no-longer contained in the boundary; in particular, $\intr{E}\subset\intr{P}_{i-1}$ and so $\bdry{F}\not\subset\bdry{P_{i-1}}$. This case is excluded since it does not correspond to a projection between the two toric varieties.
\end{enumerate}
These three possibilities are illustrated in Figure~\ref{fig:F_equivalence_reflexive_cases}.
\end{remark}

\begin{figure}[tbp]
\centering
\begin{tabular}{cp{1em}c}
\includegraphics[scale=0.5]{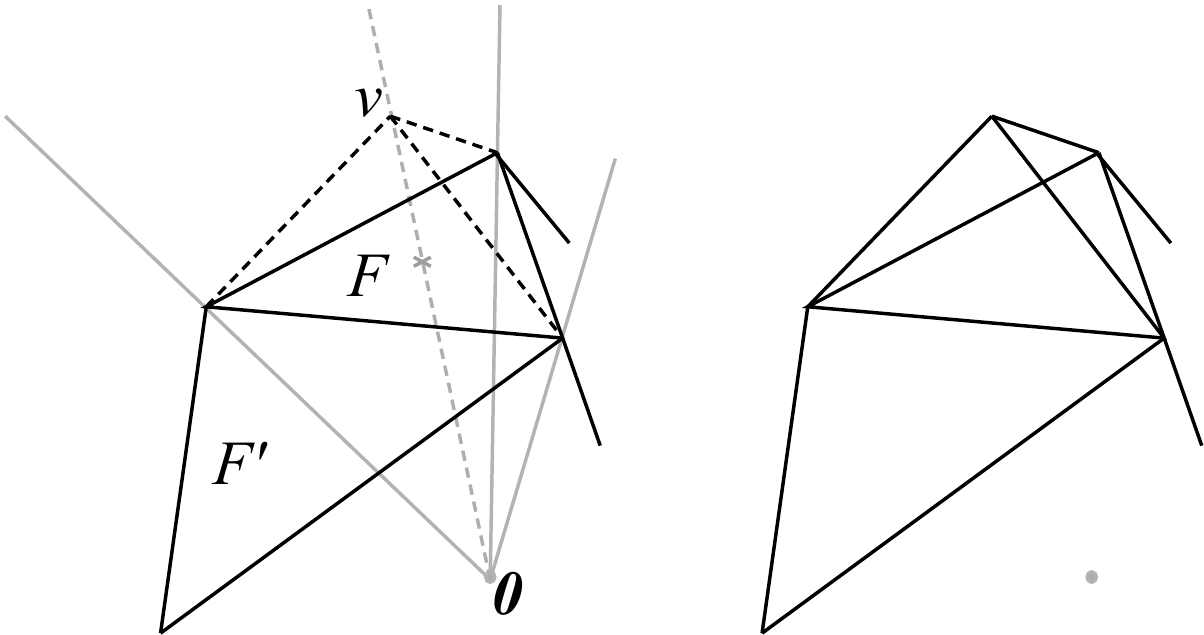}&&
\includegraphics[scale=0.5]{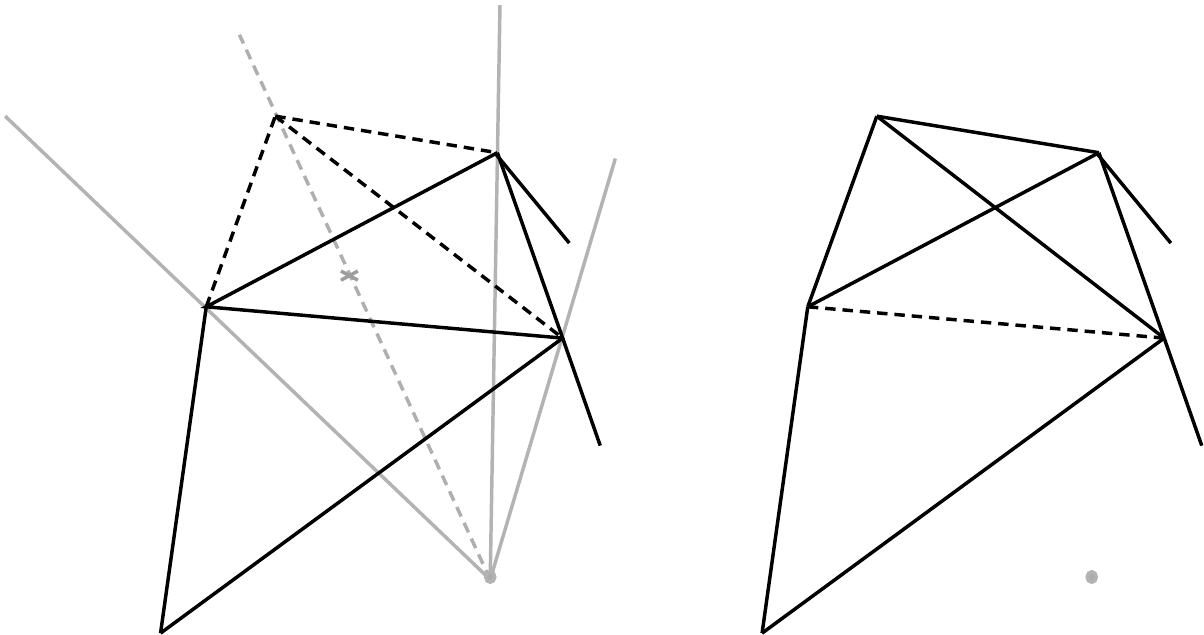}\\
(i) $u_{F'}(v)<1$&&(ii) $u_{F'}(v)=1$\\
\multicolumn{3}{c}{\includegraphics[scale=0.5]{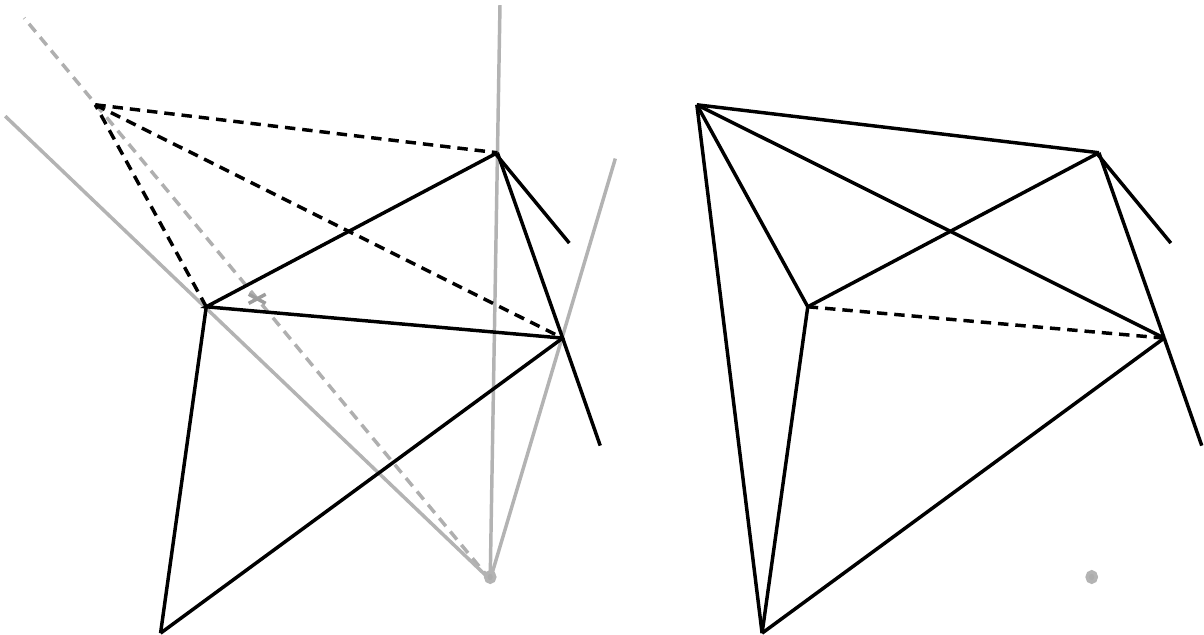}}\\
\multicolumn{3}{c}{(iii) $u_{F'}(v)>1$}
\end{tabular}
\caption{The three possible ways a facet $F'$ adjacent to $F$ can be modified when adding a new vertex $v$. See Remark~\ref{rem:F_equivalence_reflexive_cases} for an explanation.}
\label{fig:F_equivalence_reflexive_cases}
\end{figure}

\begin{corollary}[Reflexive polytopes are $F$-connected]\label{cor:F_connected_reflexive}

\noindent
\begin{enumerate}
\item
Let $P,Q\subset\NRR$ be any two three-dimensional reflexive polytopes. Then $P\Fequiv Q$.
\item\label{part:F_connected_reflexive_2}
Let $\mathcal{G}_F$ be the directed graph whose vertices are given by the three-dimensional reflexive polytopes, with an edge $P\rightarrow Q$ if and only if there exists a projection from $X_P$ to $X_Q$ (that is, if and only if $\V{P}\subset\V{Q}$, $P\Fequiv Q$, and $k=1$ in Definition~\ref{defn:F_equivalence_reflexive}). Then $\mathcal{G}_F$ has $16$ roots and $16$ sinks, and these are related via duality. The Reflexive IDs of the roots are
\[
1,\ 2,\ 3,\ 4,\ 5,\ 8,\ 9,\ 10,\ 11,\ 16,\ 18,\ 31,\ 45,\ 89,\ 102,\text{ and }105.
\]
The corresponding sinks have Reflexive IDs
\[
4312, 4282, 4318, 4284, 4287, 4310, 3314, 4313, 4315, 4299, 4238, 4251, 4303, 4319, 4309,\text{ and }4317.
\]
\end{enumerate}
\end{corollary}
\begin{proof}
This is a simple computer calculation using the classification~\cite{KS98} and Table~\ref{table:proj:faceTypes}.
\end{proof}

Corollary~\ref{cor:F_connected_reflexive} is somewhat surprising. Although we know of no reason to expect the reflexive polytopes to be $F$-connected, nor would we have expected the roots and sinks to be related via duality, we can make a small observation. Let $P,Q\subset\NRR$ be two reflexive polytopes such that there exists a finite sequence $P_0,P_1,\ldots,P_k\subset\NRR$ of reflexive polytopes satisfying conditions~\eqref{item:F_equivalence_smooth_1} and~\eqref{item:F_equivalence_smooth_2} in Definition~\ref{defn:F_equivalence} (that is, $P_{i-1}$ and $P_i$ are obtained via the addition or subtraction of a vertex). Then we say that $P$ and $Q$ are \emph{$I$-equivalent}. It is well-known that the three-dimensional reflexive polytopes are $I$-connected (although, once more, this is an experimental rather than theoretical fact). Furthermore, if one constructs a directed graph $\mathcal{G}_I$ in an analogous way to $\mathcal{G}_F$ in Corollary~\ref{cor:F_connected_reflexive}\eqref{part:F_connected_reflexive_2}, one finds the exact same list of $16$ roots and $16$ sinks (although the two graphs are different, clearly $\mathcal{G}_F$ can be included in $\mathcal{G}_I$ after possibly factoring edges). Here this duality is less mysterious: if $P$ is minimal with respect to the removal of vertices (that is, if $\conv{\V{P}\setminus\{v\}}$ is not a three-dimensional reflexive polytope for any $v\in\V{P}$) then $Q=\dual{P}$ is maximal with respect to the addition of vertices (that is, $\conv{\V{Q}\cup\{v\}}$ is not a reflexive polytope for any $v\notin Q$).

\section{Computing projections}\label{sec:projections_algorithm}
In order to do the associated computations, it is in most cases better to represent the toric degenerations of a smooth Fano variety by reflexive lattice polytopes. Given a toric degeneration, the associated reflexive Newton polytope is computed via the standard methods. However, it is worth noting that given a smooth Fano threefold $F$, it is sometimes possible to produce several different toric degenerations for $F$, resulting is different Newton polytopes, which give rise to different basic links. However, it is usually possible to connect these degenerations by a sequence of mutations (see~\cite{ACGK12}). One can choose to consider the set of degenerations of $F$ either as a whole (in order to concentrate on $F$ itself, using mutations to move between different degenerations) or as a collection of individual degenerations (to concentrate on the projections, avoiding the use of mutations).

Given a reflexive polytope, it is usually harder to determine which smooth Fano threefold it originated from. However, it is now possible due to~\cite{ACGK12,CCGK13}, where all the 3-dimensional reflexive polytopes have been listed and the possibilities for the corresponding smooth Fano threefolds have been given. 
It is worth noting that some of the polytopes do not correspond to any smooth Fano threefold even on the level of period sequences. In this paper such polytopes are disregarded, and the projections with them as intermediate steps are avoided.

In this representation, it is also possible to compute the basic links (projections) by manipulating the Newton polytopes directly.

We can take $X=\PP^{3}$ and start applying the allowed toric basic links to it. This will give us birational maps between $2868$ toric Fano threefolds with canonical Gorenstein singularities (containing representatives of $107$ different period sequences and $74$ different smooth Fano threefolds). Similarly, we can take $X$ to be any other suitable toric Fano threefold and apply the same process, getting maps between a number of toric Fano varieties. Since there are only finitely many such Fano threefolds, there must be some minimal set of projection roots that allows us to obtain all the other ones in this way. Clearly, this minimal set of roots will contain all the varieties represented by the minimal polytopes (with respect to the removal of vertices
). However, it may (depending on what projections are used) include some further toric varieties (for an example, see Remark~\ref{rmk:nonmin_roots}).

We can also look at the situation in a different way: we can consider the toric Fano threefolds with canonical Gorenstein singularities primarilly as degenerations of smooth Fano threefolds. From this point of view, the mutations (see~\S\ref{sec:mutation}) give us a second set of basic links, connecting pairs of degenerations of the same smooth Fano threefold. With this in mind, we can repeat the above process, aiming to represent all the smooth Fano threefolds (via their degenerations). One can do this with or without allowing the use of mutations. Since there are only finitely many toric Fano threefolds with canonical Gorenstein singularities, it is possible to complete all the constructions described above.

Given a starting reflexive polytope (a toric Fano threefold with canonical Gorenstein singularities), the program considers all its faces (torus-invariant points) and edges (torus-invariant curves). It selects the relevant ones (according to the rules of the ``allowed'' projections) and projects from them, obtaining a number of new polytopes (discarding those that turn out not to be reflexive). These polytopes are then added to the processing queue, taking care to avoid duplicates to the polynomials that have been found previously (two explicitly given three-dimensional polytopes are considered to be the same if one can be mapped to the other by an action of the orthogonal group on the underlying $\ZZ^3$ lattice). Such a pair of duplicates is merely an indication that there are several projection paths between a pair of varieties.

For example, there is a pair of paths between polytope $232$ (corresponding to variety $B_3$) and polytope $1969$ (corresp. to variety $\XX{2}{8}$) --- such a path can go either through
polytope $428$ (variety $B_2$) or through polytope $1599$ (variety $\XX{2}{15}$). These maps correspond to taking a pair of points (a smooth point and a cDV point) on the degeneration of $B_3$ and projectiong from them, the intermediate polytope is defined by choosing the order of projections.

Aside from obtaining new polytopes via projections from previously known ones, the program (when set to do so) also builds them by considering the possible antiprojections (i.e.\ inverses of projections) from the known polytopes. This serves a dual purpose: on the one hand, it makes it easier to explore ways of connecting several different polytopes (or the corresponding smooth Fano threefolds); on the other hand, since the procedures for the projections and the antiprojections have been written independently, this serves as an error-checking technique (if a projection $P_1\leadsto P_2$ is found, not finding an antiprojection $P_2\leadsto P_1$ of the same type would indicate an error. This is checked automatically to make sure the calculations are accurate).

Once the projection graph obtained this way is built, the program identifies all the obtained polytopes according to their ID numbers from the Graded Ring Database~\cite{GRDb}. This also makes it possible to read off other information about the polytopes, including their period sequence (bucket) numbers, which (if any) smooth Fano threefolds they are degenerations of and what mutations can be made from them.

The data gathered this way was fed into a second program for analysis. The program considered projections between reflexive polytopes and calculated the pairs of smooth Fano threefolds that these projections link (via toric degenerations). Having found such links, the program also determined the ``root'' smooth Fano threefolds and built the corresponding graph. In order to help justify the graph, it also produced examples of each of the links (trying to pick ``good'' degenerations for each threefold).

\section{Projection directed graphs}\label{sec:projection_trees}
We are interest in constructing a three-dimensional analogue of Figure~\ref{figure:del Pezzo tree}. As such, we wish to restrict our attention to those reflexive polytopes $P$ whose corresponding Gorenstein Fano variety $X_P$ is a toric degeneration of a smooth Fano threefold $X$. As a first approximation, it is reasonable to restrict our attention to those $P$ such that $\Hilb(X_P,-K_{X_P})=\Hilb(X,-K_X)$; this condition on the Hilbert series is satisfied by $4310$ of the reflexive polytopes. Even if we only allow projections passing through this subset, the result is $F$-connected.

The calculations described above have been performed, yielding the following results.
Given a toric variety $T$ let us call a projection in an anticanonical embedding from tangent space to invariant smooth point, invariant cDV point, or an invariant smooth line (see Table~\ref{table:proj:faceTypes}) F-projection.


%
%

\begin{theorem}\label{thm:proj:graph_proj-mut}
 Given any smooth Fano threefold $X$, there exists a Gorenstein toric degeneration of $X$ that can be obtained by a sequence of mutations and F-projections
 from a toric degeneration of one of $15$ smooth Fano threefolds (from now on referred to as the projections roots, see Table~\ref{table:proj:rootFanos}). 
 The directed (sub)graph connecting all Fano varieties via the projections and mutations is presented in Table~\ref{table:proj:choice}.
Each of toric degenerations we use can be equipped with a toric Landau--Ginzburg model.

 \end{theorem}
\begin{table}
 \begin{center}
  \begin{tabular}{ccl}
  Variety&Degree&Eliminated by projections from:\\
  \hline
   $\PP^3$&$64$&\\
   $\XX{2}{36}$&$62$&\\
   $Q^3$&$54$&\\
   $\XX{3}{27}$&$48$&\\
   \hline
   $\XX{2}{33}$&$54$&quartics\\
   \hline
   $\XX{3}{29}$&$50$&conics\\
   $\XX{3}{28}$&$48$&conics\\
   $\XX{4}{11}$&$42$&conics\\
   $\XX{2}{28}$&$40$&conics\\
   $\XX{3}{22}$&$40$&conics\\
   $\XX{5}{3}$&$36$&conics\\
   $\XX{6}{1}$&$30$&conics\\
   $\XX{3}{9}$&$26$&conics\\
   $\XX{7}{1}$&$24$&conics\\
   $\XX{8}{1}$&$18$&conics\\
   \hline
  \end{tabular}

 \end{center}
 \caption{Projection roots}\label{table:proj:rootFanos}
\end{table}
\begin{proof}
The list of such paths that minimize the number of projections used form the graph that can be seen in Appendix~\ref{sect:minGraph}. 
The vertices represent the Fano varieties (one vertex can represent several different degenerations of the same variety), and the arrows represent projections between degenerations of different varieties. The existence of all the arrows is shown in Appendix~\ref{sect:minGraph}
, where for each arrow an example of a relevant pair of explicit degenerations is given (according to their polytope ID from the Graded Ring Database).
\end{proof}

\begin{remark}\label{rmk:nonmin_roots}
 Not all the projection roots correspond to polytopes that are minimal  with respect to the removal of vertices: 
 for example, none of the polytopes corresponding to variety $\XX{2}{33}$ are minimal. The appearance of such roots depends on the choice of projections used --- such roots can be eliminating by allowing the use of additional projection types, like projections from curves of higher degree. In fact, a degeneration of $\XX{2}{33}$ can be obtained by projecting $\PP^{3}$ (polytope ID~$1$, minimal) from a curve of degree $4$ (obtaining polytope ID~$7$). However, as discussed above, such projections do not represent natural geometric operations on Fano varieties, and hence are not being considered.
\end{remark}

One can note that the graph above is not connected --- in fact, it has $3$ connected components. This is due to the choice of the types of projections used in the graph: using only projections from points and lines (as discussed above), the graph splits into the main component and two one-variety components (containing varieties $\XX{7}{1}$ and $\XX{8}{1}$). Variety $\XX{7}{1}$ (polytope ID~$506$) can be reached from the main component by a projection of $\XX{6}{1}$ (polytope ID~$357$) from a conic, and $\XX{8}{1}$ (polytope ID~$769$) can be obtained by projecting $\XX{7}{1}$ (polytope ID~$506$) from a conic.

\begin{theorem}
\begin{itemize}
\label{hyi}
 \item[(i)]
For any smooth Fano threefold there is its toric degeneration such that all these degenerations are connected by sequences of projections with toric centres which are either cDV points or smooth lines. The directed graph of such projections
containing toric degenerations of all smooth Fano threefolds with very ample anticanonical class
can be chosen as a union of $15$ trees with roots shown in Table~\ref{table:proj:rootFanos}.
The directed (sub)graph connecting all Fano varieties via the projections and mutations is presented in Figure~\ref{fig:graph:projections:i}.
Each of the toric degenerations can be equipped with a toric Landau--Ginzburg model. 

\item[(ii)]
For any smooth Fano threefold with very ample anticanonical class there is its Gorenstein toric degeneration such that all these degenerations are connected by sequences of projections with toric centres which are either cDV points, smooth lines, or smooth conics. The directed (Sub)graph of such projections 
can be chosen to have five roots which are: $\PP^3$, $\XX{2}{36}$, $Q^3$, $\XX{3}{27}$, and~$\XX{2}{33}$.
The directed graph connecting all Fano varieties with very ample anticanonical class via the projections and mutations is presented in Figure~\ref{fig:graph:projections:ii}.
Each of the toric degenerations can be equipped with a toric Landau--Ginzburg model.
\end{itemize}

\end{theorem}

\begin{proof}
 The algorithm described in~\S\ref{sec:projections_algorithm} 
 was implemented to find all possible projections between degenerations of smooth Fano threefolds. After that, one can choose a toric degeneration for each of the Fano varieties as described in the statement of the theorem. An example of a choice of toric degenerations that create such graphs can be seen in Table~\ref{table:proj:choice}. For each variety, it gives the corresponding Minkowski~ID, as well as the Reflexive~ID of the chosen degenerations (for the graphs in part~(i) and~(ii) of the theorem). If a variety has a terminal degeneration, its Minkowski~ID is written in bold; if the degeneration used is terminal, its Reflexive~ID is also made bold.


 Note that if the choice of polytopes from part~(i) is used in a graph that includes projections with centres in conics, one gets a directed graph with $6$ roots (the additional one being $\XX{3}{28}$). To rectify that, a different choice of degeneration needed to be made for $4$ varieties in part~(ii).
\end{proof}
\begin{table}[h]
 \caption{Degeneration choice}\label{table:proj:choice}
 \begin{center}
  \begin{tabular}{cc|cc||cc|cc||cc|cc}
Var.		&	ID	&	(i)	 	&	(ii)	 	&	Var.	&	ID	&	(i)	&	(ii)	&	Var.	&	ID	&	(i)	&	 (ii)	\\
\hline
$\PP^3$		&	$\mathbf{1}$	&	\multicolumn{2}{c||}{$\mathbf{1}$}	&	$\XX{3}{27}$	&	$\mathbf{45}$	&	 \multicolumn{2}{c||}{$\mathbf{31}$}	&	$\XX{4}{2}$	&	$110$		&	\multicolumn{2}{c}{$668$}\\
$\XX{2}{33}$	&	$\mathbf{2}$	&	\multicolumn{2}{c||}{$\mathbf{7}$}	&	$B_5$		&	$\mathbf{46}$	&	 \multicolumn{2}{c||}{$\mathbf{68}$}	&	$\XX{4}{1}$	&	$\mathbf{111}$	&	\multicolumn{2}{c}{$\mathbf{1530}$}\\
$Q^3$		&	$\mathbf{3}$	&	\multicolumn{2}{c||}{$\mathbf{4}$}	&	$\XX{4}{11}$	&	$\mathbf{48}$	&	 \multicolumn{2}{c||}{$\mathbf{85}$}	&	$\XX{3}{8}$	&	$112$		&	\multicolumn{2}{c}{$1082$}\\
$\XX{2}{30}$	&	$\mathbf{4}$	&	\multicolumn{2}{c||}{$\mathbf{23}$}	&	$\XX{3}{21}$	&	$\mathbf{49}$	&	 \multicolumn{2}{c||}{$\mathbf{214}$}	&	$V_{22}$	&	$\mathbf{113}$	&	\multicolumn{2}{c}{$\mathbf{1943}$}\\
$\XX{2}{28}$	&	$\mathbf{5}$	&	\multicolumn{2}{c||}{$\mathbf{69}$}	&	$\XX{4}{9}$	&	$\mathbf{54}$	&	 \multicolumn{2}{c||}{$\mathbf{217}$}	&	$\XX{3}{6}$	&	$117$		&	\multicolumn{2}{c}{$1501$}\\
$\XX{2}{36}$	&	$\mathbf{6}$	&	\multicolumn{2}{c||}{$\mathbf{8}$}	&	$\XX{4}{8}$	&	$\mathbf{57}$	&	 \multicolumn{2}{c||}{$\mathbf{425}$}	&	$\XX{2}{12}$	&	$\mathbf{118}$	&	\multicolumn{2}{c}{$\mathbf{2356}$}\\
$\XX{2}{35}$	&	$\mathbf{7}$	&	\multicolumn{2}{c||}{$\mathbf{6}$}	&	$\XX{2}{26}$	&	$\mathbf{58}$	&	 \multicolumn{2}{c||}{$175^{\star}$}	&	$\XX{2}{13}$	&	$119$		&	\multicolumn{2}{c}{$1924$}\\
$\XX{3}{29}$	&	$\mathbf{8}$	&	\multicolumn{2}{c||}{$\mathbf{27}$}	&	$\XX{5}{2}$	&	$\mathbf{64}$	&	 \multicolumn{2}{c||}{$\mathbf{220}$}	&	$\XX{2}{11}$	&	$120$		&	\multicolumn{2}{c}{$1701$}\\
$\XX{2}{34}$	&	$\mathbf{10}$	&	\multicolumn{2}{c||}{$\mathbf{24}$}	&	$\XX{4}{7}$	&	$\mathbf{65}$	&	 \multicolumn{2}{c||}{$\mathbf{740}$}	&	$\XX{2}{14}$	&	$122$		&	\multicolumn{2}{c}{$2353$}\\
$\XX{3}{30}$	&	$\mathbf{11}$	&	\multicolumn{2}{c||}{$\mathbf{29}$}	&	$\XX{3}{15}$	&	$\mathbf{67}$	&	 \multicolumn{2}{c||}{$\mathbf{420}$}	&	$V_{18}$	&	$124$		&	\multicolumn{2}{c}{$2703$}\\
$\XX{3}{26}$	&	$\mathbf{12}$	&	\multicolumn{2}{c||}{$\mathbf{26}$}	&	$\XX{4}{5}$	&	$\mathbf{68}$	&	 \multicolumn{2}{c||}{$\mathbf{427}$}	&	$\XX{3}{3}$	&	$135$		&	\multicolumn{2}{c}{$2678$}\\
$\XX{3}{22}$	&	$\mathbf{13}$	&	\multicolumn{2}{c||}{$\mathbf{76}$}	&	$\XX{2}{22}$	&	$\mathbf{69}$	&	 \multicolumn{2}{c||}{$373^{\star}$}	&	$\XX{7}{1}$	&	$136$		&	\multicolumn{2}{c}{$506$}\\
$\XX{3}{31}$	&	$\mathbf{14}$	&	\multicolumn{2}{c||}{$\mathbf{28}$}	&	$\XX{3}{13}$	&	$\mathbf{70}$	&	 \multicolumn{2}{c||}{$\mathbf{737}$}	&	$\XX{3}{5}$	&	$138$		&	\multicolumn{2}{c}{$1367$}\\
$\XX{2}{31}$	&	$\mathbf{15}$	&	\multicolumn{2}{c||}{$\mathbf{70}$}	&	$\XX{3}{11}$	&	$\mathbf{72}$	&	 \multicolumn{2}{c||}{$\mathbf{732}$}	&	$\XX{2}{9}$	&	$139$		&	\multicolumn{2}{c}{$3136$}\\
$\XX{3}{25}$	&	$\mathbf{16}$	&	\multicolumn{2}{c||}{$\mathbf{74}$}	&	$\XX{2}{18}$	&	$74$		&	\multicolumn{2}{c||}{$1090$}		 &	$B_2$		&	$140$		&	\multicolumn{2}{c}{$428$}\\
$\XX{3}{23}$	&	$\mathbf{17}$	&	\multicolumn{2}{c||}{$\mathbf{205}$}	&	$B_4$		&	$\mathbf{75}$	&	 \multicolumn{2}{c||}{$154^{\star}$}	&	$\XX{3}{4}$	&	$142$		&	\multicolumn{2}{c}{$2222$}\\
$\XX{3}{19}$	&	$\mathbf{18}$	&	\multicolumn{2}{c||}{$\mathbf{206}$}	&	$\XX{5}{3}$	&	$\mathbf{76}$	&	 \multicolumn{2}{c||}{$\mathbf{219}$}	&	$V_{16}$	&	$143$		&	\multicolumn{2}{c}{$2482$}\\
$\XX{2}{27}$	&	$\mathbf{19}$	&	\multicolumn{2}{c||}{$\mathbf{201}$}	&	$\XX{2}{23}$	&	$\mathbf{78}$	&	 \multicolumn{2}{c||}{$\mathbf{411}$}	&	$\XX{2}{8}$	&	$144$		&	\multicolumn{2}{c}{$1969$}\\
$\XX{3}{14}$	&	$\mathbf{21}$	&	\multicolumn{2}{c||}{$\mathbf{203}$}	&	$\XX{4}{6}$	&	$\mathbf{81}$	&	 \multicolumn{2}{c||}{$\mathbf{426}$}	&	$\XX{2}{10}$	&	$145$		&	\multicolumn{2}{c}{$3036$}\\
$\XX{3}{9}$	&	$22$		&	\multicolumn{2}{c||}{$374$}		&	$\XX{4}{4}$	&	$\mathbf{83}$	&	\multicolumn{2}{c||}{$\mathbf{741}$}	&	 $V_{14}$	&	$147$		&	\multicolumn{2}{c}{$3283$}\\
$\XX{2}{32}$	&	$\mathbf{24}$	&	\multicolumn{2}{c||}{$\mathbf{22}$}	&	$\XX{2}{21}$	&	$\mathbf{84}$	&	 \multicolumn{2}{c||}{$\mathbf{731}$}	&	$\XX{2}{7}$	&	$148$		&	\multicolumn{2}{c}{$3239$}\\
$\XX{3}{28}$	&	$\mathbf{28}$	&	$\mathbf{30}$	&	$\mathbf{81}$	&	$\XX{3}{12}$	&	$\mathbf{85}$	&	 \multicolumn{2}{c||}{$723^{\star}$}	&	$\XX{2}{6}$	&	$149$		&	\multicolumn{2}{c}{$3319$}\\
$\XX{4}{13}$	&	$\mathbf{29}$	&	\multicolumn{2}{c||}{$\mathbf{84}$}	&	$\XX{2}{19}$	&	$\mathbf{86}$	&	 \multicolumn{2}{c||}{$\mathbf{1109}$}	&	$V_{12}$	&	$150$		&	\multicolumn{2}{c}{$3966$}\\
$\XX{3}{24}$	&	$\mathbf{31}$	&	\multicolumn{2}{c||}{$\mathbf{78}$}	&	$\XX{2}{20}$	&	$\mathbf{87}$	&	 \multicolumn{2}{c||}{$\mathbf{1110}$}	&	$\XX{3}{1}$	&	$154$		&	\multicolumn{2}{c}{$3350$}\\
$\XX{4}{12}$	&	$\mathbf{34}$	&	$\mathbf{83}$	&	$190^{\star}$	&	$\XX{4}{3}$	&	$\mathbf{88}$	&	 \multicolumn{2}{c||}{$\mathbf{735}$}	&	$\XX{8}{1}$	&	$155$		&	\multicolumn{2}{c}{$769$}\\
$\XX{2}{29}$	&	$\mathbf{35}$	&	$\mathbf{72}$	&	$\mathbf{204}$	&	$\XX{3}{10}$	&	$\mathbf{99}$	&	 \multicolumn{2}{c||}{$\mathbf{1113}$}	&	$\XX{3}{2}$	&	$157$		&	\multicolumn{2}{c}{$2791$}\\
$\XX{4}{10}$	&	$\mathbf{37}$	&	\multicolumn{2}{c||}{$\mathbf{215}$}	&	$\XX{5}{1}$	&	$100$		&	\multicolumn{2}{c||}{$673$}		&	 $\XX{2}{5}$	&	$158$		&	\multicolumn{2}{c}{$3453$}\\
$\XX{3}{20}$	&	$\mathbf{38}$	&	\multicolumn{2}{c||}{$\mathbf{80}$}	&	$\XX{2}{17}$	&	$\mathbf{101}$	&	 \multicolumn{2}{c||}{$\mathbf{1528}$}	&	$V_{10}$	&	$160$		&	\multicolumn{2}{c}{$4132$}\\
$\XX{3}{17}$	&	$\mathbf{39}$	&	\multicolumn{2}{c||}{$\mathbf{210}$}	&	$\XX{3}{7}$	&	$\mathbf{103}$	&	 \multicolumn{2}{c||}{$\mathbf{1529}$}	&	$\XX{2}{4}$	&	$161$		&	\multicolumn{2}{c}{$4031$}\\
$\XX{3}{18}$	&	$\mathbf{41}$	&	$\mathbf{212}$	&	$\mathbf{419}$	&	$\XX{2}{16}$	&	$104$		&	\multicolumn{2}{c||}{$1485$}		 &	$V_8$		&	$163$		&	\multicolumn{2}{c}{$4205$}\\
$\XX{3}{16}$	&	$\mathbf{42}$	&	\multicolumn{2}{c||}{$\mathbf{418}$}	&	$B_3$		&	$106$		&	\multicolumn{2}{c||}{$232$}		&	 $V_6$		&	$164$		&	\multicolumn{2}{c}{$4286$}\\
$\XX{2}{25}$	&	$\mathbf{43}$	&	\multicolumn{2}{c||}{$\mathbf{410}$}	&	$\XX{6}{1}$	&	$107$		&	\multicolumn{2}{c||}{$357$}		&	 $V_4$		&	$165$		&	\multicolumn{2}{c}{$4312$}\\
$\XX{2}{24}$	&	$\mathbf{44}$	&	\multicolumn{2}{c||}{$\mathbf{412}$}	&	$\XX{2}{15}$	&	$109$		&	\multicolumn{2}{c||}{$1599$}		 &	$ $		&	$ $		&	\multicolumn{2}{c}{$ $}

  \end{tabular}

 \end{center}
\end{table}

%
%
%

\begin{remark}
 These graphs are not the only ones that satisfy the statement of the theorem --- for many smooth Fano threefolds, a choice of degeneration needed to be made. In this case, the choice was made according to the following priorities:
 \begin{enumerate}
  \item to minimize the number of the graphs' roots;
  \item to minimize the number of the connected components of the graphs;
  \item where possible, to use a terminal degeneration of the variety.
 \end{enumerate}
 In the choices above, the numbers of roots and connected components are indeed minimal, and a terminal degeneration was used where available except in $4$ cases ($5$ cases if projections with centres in smooth conics were used), marked by $\star$, where this would lead to getting additional roots or connected components. The degenerations in question are those for varieties $B_4$, $\XX{2}{22}$, $\XX{2}{26}$, $\XX{3}{12}$ (and~$\XX{4}{12}$ if projections with centres at conics were used).
\end{remark}

\appendix
\section{Projection--minimizing graph}\label{sect:minGraph}
%
%

Table~\ref{table:projection graph} contains a choice of degenerations $(X_1,X_2)$ for every arrow in the graph of projections discussed in Theorem~\ref{thm:proj:graph_proj-mut}.
The ``From'' column gives the starting variety of the projection along with the variety's degree and Minkowski ID. The ``To'' column lists all the possible destination varieties, and, for each of them, an example of the corresponding pair of degenerations $(X_1,X_2)$ (in terms of their Reflexive ID's).
\begin{table}[h]
 \caption{Projection graph}\label{table:projection graph}
\begin{tabular}{l||lll||@{\ \ \ \ }r@{:\ \ (}rl@{)\ \ \ \ \ \ \ }r@{:\ \ (}rl@{)\ \ \ \ \ \ \ }r@{:\ \ (}rl@{)}}
\multicolumn{4}{c}{From:}&\multicolumn{9}{c}{To:}\\
Level&Var.&P.S.&Deg.&\multicolumn{9}{c}{}\\
\hline
0&1.1&1&64&2.35&1&6&\multicolumn{6}{c}{ }\\
0&2.36&6&62&2.34&8&24&\multicolumn{6}{c}{ }\\
0&1.2&3&54&3.31&2&20&2.30&2&14&\multicolumn{3}{c}{ }\\
0&2.33&2&54&3.30&7&29&3.26&7&26&\multicolumn{3}{c}{ }\\
0&3.29&8&50&3.26&27&73&3.24&27&78&\multicolumn{3}{c}{ }\\
0&3.27&45&48&1.7&18&43&\multicolumn{6}{c}{ }\\
0&3.28&28&48&4.12&52&181&2.29&52&106&\multicolumn{3}{c}{ }\\
0&4.11&48&42&3.21&62&184&4.9&191&291&2.26&62&163\\
0&2.28&5&40&2.27&34&305&3.14&34&143&\multicolumn{3}{c}{ }\\
0&3.22&13&40&3.19&139&270&3.17&76&210&3.15&139&316\\
0&5.3&76&36&4.6&195&409&2.21&114&238&\multicolumn{3}{c}{ }\\
0&6.1&107&30&4.2&284&602&\multicolumn{6}{c}{ }\\
0&3.9&22&26&2.18&447&1999&\multicolumn{6}{c}{ }\\
0&7.1&136&24&\multicolumn{9}{c}{ }\\
0&8.1&155&18&\multicolumn{9}{c}{ }\\
\hline
1&2.35&7&56&2.32&6&13&\multicolumn{6}{c}{ }\\
1&2.34&10&54&2.31&5&21&\multicolumn{6}{c}{ }\\
1&3.31&14&52&3.25&20&47&\multicolumn{6}{c}{ }\\
1&3.30&11&50&4.13&29&60&\multicolumn{6}{c}{ }\\
1&2.30&4&46&3.25&14&47&3.23&23&77&\multicolumn{3}{c}{ }\\
1&3.26&12&46&3.23&26&77&3.20&26&44&\multicolumn{3}{c}{ }\\
1&4.12&34&44&4.10&309&638&3.18&309&639&\multicolumn{3}{c}{ }\\
1&3.24&31&42&4.10&169&390&3.20&169&259&\multicolumn{3}{c}{ }\\
1&1.7&46&40&1.6&221&429&\multicolumn{6}{c}{ }\\
1&2.29&35&40&3.20&106&259&3.18&19&63&2.25&106&251\\
1&2.27&19&38&3.18&71&212&3.16&71&213&2.24&157&322\\
1&3.19&18&38&3.16&75&213&3.13&35&93&\multicolumn{3}{c}{ }\\
1&3.21&49&38&4.8&488&1350&2.23&488&765&\multicolumn{3}{c}{ }\\
1&4.9&54&38&4.8&291&580&4.7&67&185&2.22&291&547\\
1&3.17&39&36&3.16&130&323&\multicolumn{6}{c}{ }\\
1&2.26&58&34&2.22&1263&1749&2.23&481&765&2.19&957&1230\\
1&3.14&21&32&2.24&143&631&3.11&308&656&3.7&143&263\\
1&3.15&67&32&2.22&2125&2450&3.13&564&977&4.4&598&1355\\
&&&&3.11&564&985&3.12&954&1435&2.17&564&931\\
1&4.6&81&32&4.4&1682&2179&3.12&1682&2110&4.3&1682&1976\\
&&&&2.17&1708&2182&\multicolumn{6}{c}{ }\\
1&2.21&84&28&2.19&123&338&2.20&238&1320&3.10&123&267\\
&&&&2.17&1753&2252&2.12&2729&3077&\multicolumn{3}{c}{ }\\
1&4.2&110&26&3.8&602&1867&1.17&602&930&3.6&1048&1437\\
1&2.18&74&24&2.16&449&1998&2.13&1441&1824&\multicolumn{3}{c}{ }\\
\hline
\end{tabular}
\end{table}

\begin{table}[h]
\begin{tabular}{llll||@{\ \ \ \ }r@{:\ \ (}rl@{)\ \ \ \ \ \ \ }r@{:\ \ (}rl@{)\ \ \ \ \ \ \ }r@{:\ \ (}rl@{)}}
&\multicolumn{3}{c}{From:}&\multicolumn{9}{c}{}\\
Level&Var.&Deg.&P.S.&\multicolumn{9}{c}{Projection to}\\
\hline
2&2.32&24&48&\multicolumn{9}{c}{ }\\
2&2.31&15&46&\multicolumn{9}{c}{ }\\
2&4.13&29&46&\multicolumn{9}{c}{ }\\
2&3.25&16&44&\multicolumn{9}{c}{ }\\
2&3.23&17&42&\multicolumn{9}{c}{ }\\
2&4.10&37&40&5.2&180&408&\multicolumn{6}{c}{ }\\
2&3.20&38&38&5.2&44&194&\multicolumn{6}{c}{ }\\
2&3.18&41&36&4.5&63&151&\multicolumn{6}{c}{ }\\
2&4.8&57&36&\multicolumn{9}{c}{ }\\
2&3.16&42&34&4.5&615&981&\multicolumn{6}{c}{ }\\
2&4.7&65&34&\multicolumn{9}{c}{ }\\
2&1.6&75&32&1.5&433&742&\multicolumn{6}{c}{ }\\
2&2.25&43&32&\multicolumn{9}{c}{ }\\
2&2.22&69&30&\multicolumn{9}{c}{ }\\
2&2.23&78&30&2.15&1165&1599&\multicolumn{6}{c}{ }\\
2&2.24&44&30&\multicolumn{9}{c}{ }\\
2&3.13&70&30&5.1&93&285&\multicolumn{6}{c}{ }\\
2&4.4&83&30&\multicolumn{9}{c}{ }\\
2&3.11&72&28&\multicolumn{9}{c}{ }\\
2&3.12&85&28&\multicolumn{9}{c}{ }\\
2&4.3&88&28&\multicolumn{9}{c}{ }\\
2&2.19&86&26&2.15&2066&2397&2.11&1392&1701&\multicolumn{3}{c}{ }\\
2&2.20&87&26&1.16&1186&1559&\multicolumn{6}{c}{ }\\
2&3.10&99&26&4.1&267&489&1.16&3571&3764&\multicolumn{3}{c}{ }\\
2&2.17&101&24&2.9&2182&2462&\multicolumn{6}{c}{ }\\
2&3.7&103&24&2.14&1374&1721&\multicolumn{6}{c}{ }\\
2&3.8&112&24&2.14&2541&2958&3.5&984&1367&\multicolumn{3}{c}{ }\\
2&1.17&113&22&2.14&1141&1659&3.5&930&1367&1.16&4037&4085\\
&&&&1.14&3772&3891&\multicolumn{6}{c}{ }\\
2&2.16&104&22&1.16&2960&3233&2.11&1233&1701&\multicolumn{3}{c}{ }\\
2&3.6&117&22&2.14&3123&3921&1.16&2166&2494&3.3&3123&3445\\
&&&&3.4&2166&2394&\multicolumn{6}{c}{ }\\
2&2.12&118&20&1.16&1548&2500&2.11&1194&1701&2.9&2311&2606\\
&&&&2.6&3077&3319&\multicolumn{6}{c}{ }\\
2&2.13&119&20&1.16&3557&4085&3.3&1662&2070&3.4&2114&2544\\
&&&&1.15&1717&2024&2.9&2770&3136&\multicolumn{3}{c}{ }\\
\hline
\end{tabular}
\end{table}

\begin{table}[h]
\begin{tabular}{llll||@{\ \ \ \ }r@{:\ \ (}rl@{)\ \ \ \ \ \ \ }r@{:\ \ (}rl@{)\ \ \ \ \ \ \ }r@{:\ \ (}rl@{)}}
&\multicolumn{3}{c}{From:}&\multicolumn{9}{c}{}\\
Level&Var.&Deg.&P.S.&\multicolumn{9}{c}{Projection to}\\
\hline
3&5.2&64&36&\multicolumn{9}{c}{ }\\
3&4.5&68&32&\multicolumn{9}{c}{ }\\
3&5.1&100&28&\multicolumn{9}{c}{ }\\
3&1.5&106&24&1.4&1952&2364&\multicolumn{6}{c}{ }\\
3&4.1&111&24&\multicolumn{9}{c}{ }\\
3&2.15&109&22&2.8&1599&1969&\multicolumn{6}{c}{ }\\
3&2.14&122&20&2.10&3921&4002&\multicolumn{6}{c}{ }\\
3&3.5&138&20&\multicolumn{9}{c}{ }\\
3&1.16&124&18&2.10&3361&3617&\multicolumn{6}{c}{ }\\
3&2.11&120&18&2.8&1701&1969&\multicolumn{6}{c}{ }\\
3&3.3&135&18&2.7&3445&3592&\multicolumn{6}{c}{ }\\
3&3.4&142&18&2.10&2394&2746&3.2&2544&2791&\multicolumn{3}{c}{ }\\
3&1.15&143&16&2.7&3373&3592&3.2&2482&2791&1.13&4050&4119\\
3&2.9&139&16&\multicolumn{9}{c}{ }\\
3&1.14&147&14&1.13&3887&4119&2.5&3587&3736&1.12&4171&4200\\
3&2.6&149&12&1.12&3319&4007&\multicolumn{6}{c}{ }\\
\hline
4&1.4&140&16&\multicolumn{9}{c}{ }\\
4&2.10&145&16&\multicolumn{9}{c}{ }\\
4&2.7&148&14&3.1&3102&3329&\multicolumn{6}{c}{ }\\
4&2.8&144&14&\multicolumn{9}{c}{ }\\
4&3.2&157&14&\multicolumn{9}{c}{ }\\
4&1.13&150&12&2.4&3966&4031&\multicolumn{6}{c}{ }\\
4&2.5&158&12&2.4&3736&4031&\multicolumn{6}{c}{ }\\
4&1.12&160&10&1.11&3051&3314&\multicolumn{6}{c}{ }\\
\hline
5&3.1&154&12&\multicolumn{9}{c}{ }\\
5&2.4&161&10&\multicolumn{9}{c}{ }\\
5&1.11&163&8&\multicolumn{9}{c}{ }\\
\hline
\end{tabular}
\end{table}



\begin{figure}
 \caption{Connecting varieties by projections from points and lines}\label{fig:graph:projections:i}
 \begin{center}
  The Fano threefolds are denoted by their Minkowski ID's, each arrow signify a projection between them. See Table~\ref{table:proj:choice}, (i) for the explicit choice of degenerations for each of the varieties.

  \includegraphics[width=15cm]{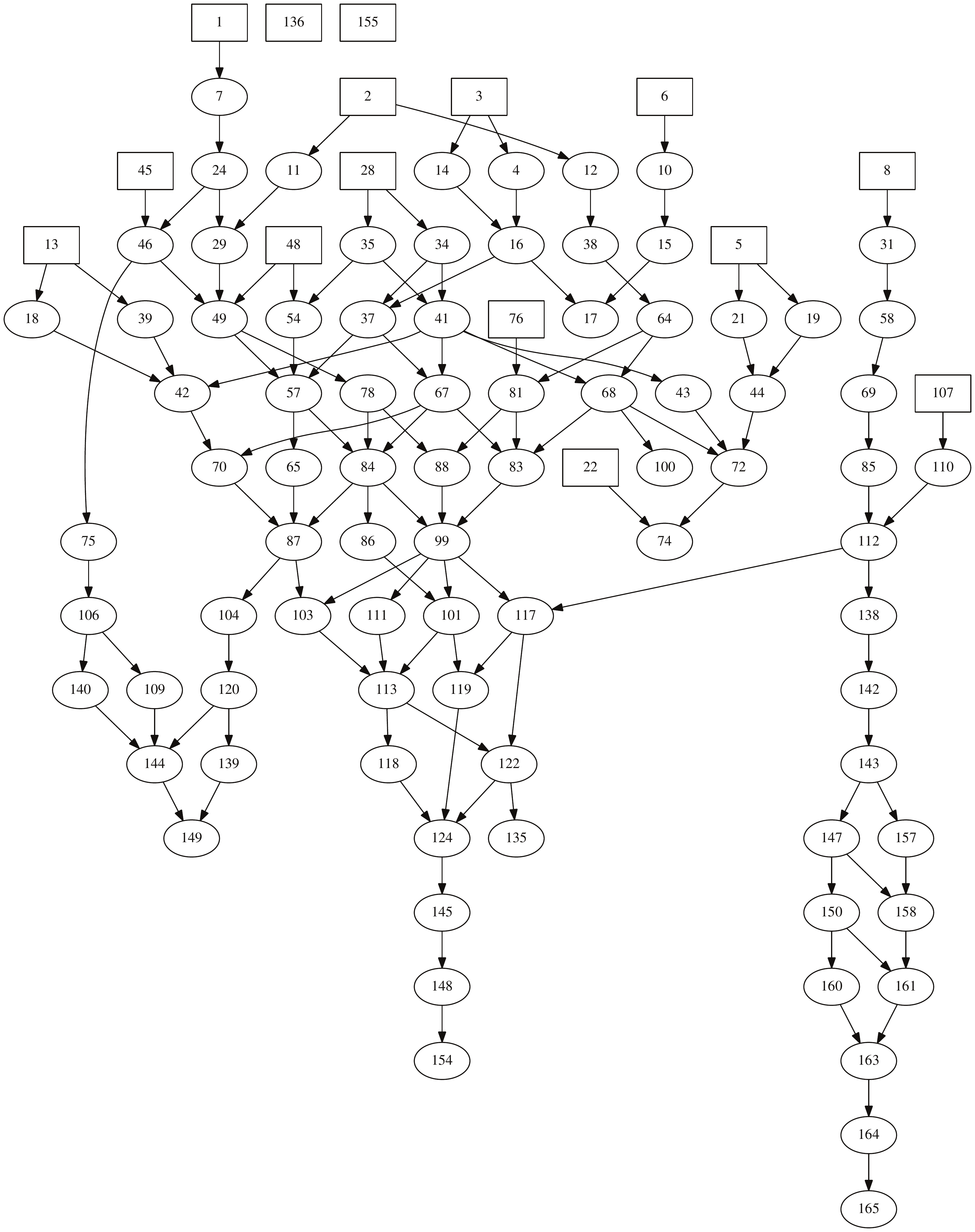}
  \end{center}
\end{figure}

\begin{figure}
 \caption{Connecting varieties by projections from points, lines, and conics}\label{fig:graph:projections:ii}
 \begin{center}
  The Fano threefolds are denoted by their Minkowski ID's, each arrow signify a projection between them. See Table~\ref{table:proj:choice}, (ii) for the explicit choice of degenerations for each of the varieties.

  \includegraphics[width=15cm]{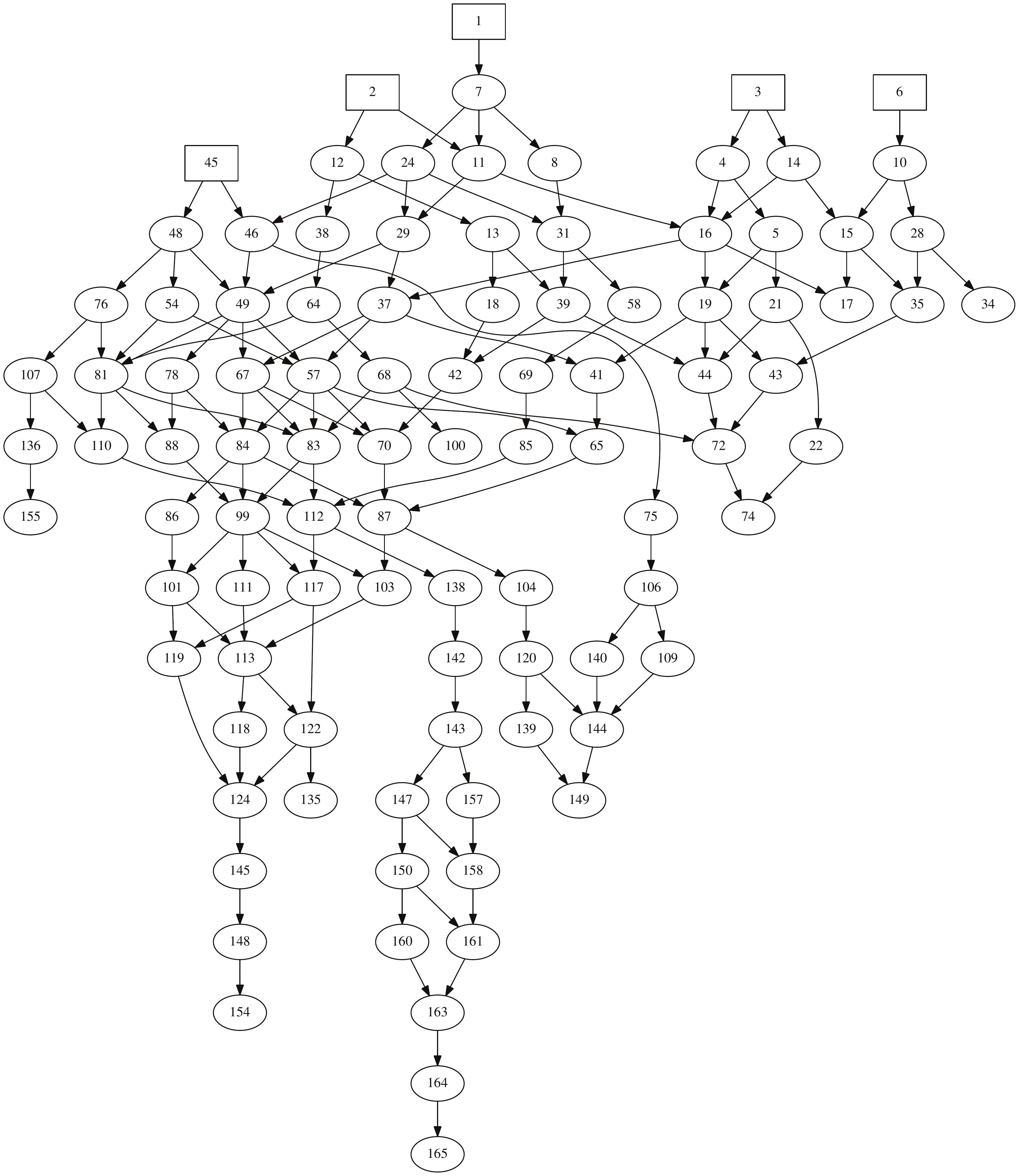}
  \end{center}
\end{figure}

\clearpage
\section{Tables for Toric Degenerations}\label{sec:appendix}

\section{Hilbert Scheme Component Dimensions for Smooth Fano Threefolds}\label{app:compdims}
In the following table, we compute $h^0(V,\mathcal N_{V/\PP^n})$ for each smooth Fano threefold with very ample anticanonical divisor. Here $V\hookrightarrow\PP^n$ denotes the anticanonical embedding. This value equals the dimension of the Hilbert scheme component of $\mathcal{hilb}_V$ corresponding to $V$.
\noindent
\begin{longtable}{c c c c c}
\toprule
Name & $b_2$ & $b_3/2$ & $(-K_V)^{3}$ & $h^0(V,\mathcal N_{V/\PP^n})$\\\midrule
\endhead
\bottomrule
\endfoot
\oddrow$V_4$&$1$&$30$&$4$&$69$\\\midrule
\evnrow$V_6$&$1$&$20$&$6$&$69$\\\midrule
\oddrow$V_8$&$1$&$14$&$8$&$75$\\\midrule
\evnrow$V_{10}$&$1$&$10$&$10$&$85$\\\midrule
\oddrow$V_{12}$&$1$&$7$&$12$&$98$\\\midrule
\evnrow$V_{14}$&$1$&$5$&$14$&$114$\\\midrule
\oddrow$V_{16}$&$1$&$3$&$16$&$132$\\\midrule
\evnrow$V_{18}$&$1$&$2$&$18$&$153$\\
\end{longtable}

\newpage
\section{Tangent Space Dimensions for Gorenstein Toric  Fano Threefolds}\label{app:tsdims}
In the following table, we record the tangent space dimension $h^0(\mathcal N)$ of the Hilbert scheme points corresponding to anticanonically-embedded Gorenstein toric Fano threefolds.

\begin{multicols}{3}

\setbox\ltmcbox\vbox{
\makeatletter\col@number\@ne
\begin{longtable}{@{\ }c c c c c@{\ }}
\toprule
\small{Reflexive ID} & $h^0(\mathcal N)$\\\midrule
\endhead
\bottomrule
\endfoot
\oddrow$1$&$1209$\\\midrule
\evnrow$2$&$889$\\\midrule
\oddrow$3$&$1501$\\\midrule
\evnrow$4$&$889$\\\midrule
\oddrow$5$&$888$\\\midrule
\evnrow$6$&$948$\\\midrule
\oddrow$7$&$888$\\\midrule
\evnrow$8$&$1140$\\
\end{longtable}
\unskip
\unpenalty
\unpenalty}

\unvbox\ltmcbox

\end{multicols}

\subsection{Small Toric Degenerations for Fano Threefolds}\label{app:small}
Below, we record all degenerations of smooth Fano threefolds to Gorenstein toric Fano varieties with at most terminal singularities \cite{Gal07}.

\begin{multicols}{3}

\setbox\ltmcbox\vbox{
\makeatletter\col@number\@ne

\begin{longtable}{@{\ }cl@{\ }}
\toprule
Name &  Reflexive IDs \\\midrule
\endhead
\bottomrule
\endfoot
\oddrow$V_{22}$&{1943}\\\midrule
\evnrow$B_4$&{198}\\\midrule
\oddrow$B_5$&{68}\\\midrule
\evnrow$Q^3$&{4}\\\midrule
\oddrow$\PP^3$&{1}\\\midrule
\evnrow$2\text{--}12$&{2356}\\\midrule
\oddrow$2\text{--}17$&{1528}\\\midrule
\evnrow$2\text{--}19$&{1109}\\
\end{longtable}

\unskip
\unpenalty
\unpenalty}

\unvbox\ltmcbox

\end{multicols}

\subsection{Select Toric Degenerations for Low Degree Fano Threefolds}\label{app:lowdeg}
Below we record degenerations of products of low degree Fano threefolds as in \cite{CI14}. We only list those degenerations with smooth deformation space.
\begin{longtable}{c p{10cm}}
\toprule
Name &  Reflexive IDs \\\midrule
\endhead
\bottomrule
\endfoot
\oddrow$V_4$&{4312}\\\midrule
\evnrow$V_6$&{4282, 4284, 4286, 4287, 4297, 4310, 4318}\\\midrule
\oddrow$V_8$&{3314, 4005, 4167, 4194, 4203, 4205, 4217, 4231, 4238, 4244, 4250, 4251, 4267, 4269, 4275, 4280, 4290, 4298, 4299, 4304, 4313, 4314, 4315}\\\midrule
\evnrow$V_{10}$&{3051, 3792, 3903, 3922, 3928, 3965, 4007, 4023, 4032, 4044, 4074, 4118, 4131, 4132, 4135, 4143, 4161, 4168, 4181, 4186, 4190, 4200, 4202, 4214, 4216, 4219, 4225, 4228, 4245, 4252, 4255, 4258, 4261, 4268, 4281, 4291, 4292, 4295, 4301, 4303, 4307, 4311}\\\midrule
\oddrow$V_{12}$&{2756, 3406, 3625, 3626, 3667, 3683, 3702, 3796, 3848, 3853, 3869, 3937, 3938, 3946, 3962, 3966, 3981, 3984, 3985, 3991, 4059, 4060, 4070, 4080, 4102, 4104, 4119, 4122, 4124, 4133, 4134, 4144, 4145, 4179, 4180, 4184, 4218, 4220, 4247, 4254, 4263, 4270, 4272, 4273, 4293, 4294}\\\midrule
\evnrow$2\text{--}4$&{4031}\\\midrule
\oddrow$2\text{--}5$&{3453, 3736, 3777}\\\midrule
\evnrow$2\text{--}6$&{3319, 3349, 3416}\\\midrule
\oddrow$3\text{--}1$&{3329, 3350}\\\midrule
\end{longtable}

\subsection{Select Toric Degenerations for Rank One Index One Fano Threefolds}\label{app:tri}
Below we record select degenerations of rank one index one Fano threefolds as described in \S \ref{sec:triang}.
\begin{longtable}{c p{10cm}}
\toprule
Name &  Reflexive IDs \\\midrule
\endhead
\bottomrule
\endfoot
\oddrow$V_{14}$&{2464, 3214, 3229, 3268, 3298, 3461, 3506, 3535, 3536, 3633, 3670, 3675, 3676, 3712, 3719, 3724, 3794, 3837, 3849, 3852, 3870, 3876, 3887, 3891, 3944, 3968, 3970, 3982, 3986, 3992, 4046, 4066, 4069, 4077, 4078, 4139, 4140, 4171, 4185, 4223, 4226, 4253, 4271}\\\midrule
\evnrow$V_{16}$&{2024, 2498, 2650, 2894, 2899, 2903, 2987, 2990, 3026, 3027, 3034, 3112, 3162, 3210, 3277, 3280, 3294, 3300, 3310, 3373, 3374, 3396, 3464, 3513, 3515, 3519, 3546, 3642, 3648, 3674, 3685, 3700, 3709, 3710, 3711, 3721, 3725, 3742, 3797, 3798, 3816, 3858, 3864, 3872, 3954, 3975, 3990, 3993, 4050, 4063, 4065, 4109, 4121, 4146, 4150, 4172, 4227}\\\midrule
\oddrow$V_{18}$&{1559, 1827, 2116, 2187, 2223, 2289, 2302, 2337, 2480, 2494, 2500, 2520, 2539, 2605, 2646, 2648, 2658, 2683}\\
\end{longtable}

\subsection{Select Toric Degenerations for Product Fano Threefolds}\label{app:prod}
Below we record degenerations of products of del Pezzo surfaces with $\PP^1$.
\begin{longtable}{c l}
\toprule
Name &  Reflexive ID \\\midrule
\endhead
\bottomrule
\endfoot
\oddrow$5\text{--}3$&{219}\\\midrule
\evnrow$6\text{--}1$&{357}\\\midrule
\oddrow$7\text{--}1$&{510}\\
\end{longtable}

\subsection{Select Toric Degenerations for Complete Intersections in Toric Varieties}\label{app:ci}
We list degenerations constructed as in \S \ref{sec:ci}.

\begin{longtable}{c c l}
\toprule
Name & Reflexive ID & Regular Sequence\\\midrule
\endhead
\bottomrule
\endfoot
\oddrow&1978&{${x}_{2}^{2} {x}_{4}^{2}-{x}_{1} {x}_{3} {x}_{5}$}\\
\oddrow&2364&{${x}_{2}^{3} {x}_{4}-{x}_{1} {x}_{3} {x}_{5}$}\\
\oddrow\multirow{-3}{*}{$B_2$}&3313&{${x}_{2}^{4}-{x}_{1} {x}_{3} {x}_{5}$}\\\midrule

\evnrow&232&{${x}_{1} {x}_{2} {x}_{4}-{x}_{3} {x}_{5}^{2}$}\\
\evnrow\multirow{-2}{*}{$B_3$}&742&{${x}_{1} {x}_{2} {x}_{3}-{x}_{5}^{3}$}\\\midrule

\oddrow$2\text{--}7$&3813&{${x}_{1} {x}_{3} {x}_{4}-{x}_{2} {x}_{5} {x}_{6}$, ${x}_{6}^{2}-{x}_{5} {x}_{7}$}\\\midrule

\evnrow$2\text{--}8$&1969&{${x}_{1} {x}_{2} {x}_{3} {x}_{4} {x}_{5}-{x}_{6}^{2}$}\\\midrule

\oddrow&2606&{${x}_{2} {x}_{5}-{x}_{4} {x}_{7}$, ${x}_{1}^{2} {x}_{5}-{x}_{2} {x}_{3} {x}_{6}$}\\
\end{longtable}

\subsection{Further Toric Degenerations for Smooth  Fano Threefolds}\label{app:mutation}
In the following table, we list degenerations of smooth Fano threefolds to Gorenstein toric Fano threefolds obtained from those of \S \ref{app:small}, \ref{app:lowdeg}, \ref{app:tri}, \ref{app:prod}, \ref{app:ci} by (possibly repeated) mutations. Those toric varieties corresponding to a smooth point of $\calH_V$ appear in the column ``Interior Points''; those corresponding to singular points of $\calH_V$ appear in the column ``Boundary Points''.

\begin{longtable}{c p{4cm}@{\hspace{1cm}} p{6cm}}
\toprule
Name & Interior Points & Boundary Points\\\midrule
\endhead
\bottomrule
\endfoot
\oddrow$V_4$&{4312}&{}\\\midrule
\evnrow$V_6$&{4282, 4284, 4286, 4287, 4297, 4310, 4318}&{}\\\midrule
\oddrow$V_8$&{3314, 4005, 4167, 4194, 4203, 4205, 4217, 4231, 4238, 4244, 4250, 4251, 4267, 4269, 4275, 4280, 4290, 4298, 4299, 4304, 4313, 4314, 4315}&{}\\\midrule
\evnrow$V_{10}$&{3051, 3792, 3903, 3922, 3928, 3965, 4007, 4023, 4032, 4044, 4074, 4118, 4131, 4132, 4135, 4143, 4161, 4168, 4181, 4186, 4190, 4200, 4202, 4214, 4216, 4219, 4225, 4228, 4245, 4252, 4255, 4258, 4261, 4268, 4281, 4291, 4292, 4295, 4301, 4303, 4307, 4311}&{3879, 3927, 3964, 4006, 4024, 4042}\\\midrule
\oddrow$V_{12}$&{2756, 3406, 3625, 3626, 3667, 3683, 3702, 3796, 3848, 3853, 3869, 3937, 3938, 3946, 3962, 3966, 3981, 3984, 3985, 3991, 4059, 4060, 4070, 4080, 4102, 4104, 4119, 4122, 4124, 4133, 4134, 4144, 4145, 4179, 4180, 4184, 4218, 4220, 4247, 4254, 4263, 4270, 4272, 4273, 4293, 4294}&{3452, 3756, 3760, 3795, 3846, 3857, 3875, 3933, 3936, 3967, 3983, 4004, 4027, 4041, 4043, 4075, 4117, 4241, 4249}\\\midrule
\evnrow$V_{14}$&{2464, 3082, 3214, 3219, 3229, 3268, 3283, 3298, 3461, 3471, 3506, 3535, 3536, 3587, 3633, 3637, 3640, 3670, 3675, 3676, 3712, 3719, 3724, 3794, 3837, 3849, 3852, 3870, 3876, 3887, 3891, 3932, 3944, 3968, 3970, 3982, 3986, 3992, 4046, 4066, 4069, 4077, 4078, 4139, 4140, 4171, 4185, 4198, 4223, 4226, 4253, 4271}&{2472, 2553, 2725, 2966, 3056, 3202, 3266, 3397, 3399, 3509, 3511, 3530, 3537, 3543, 3544, 3559, 3574, 3578, 3624, 3636, 3666, 3698, 3704, 3718, 3766, 3799, 3800, 3824, 3825, 3832, 3847, 3855, 3867, 3871, 3892, 3896, 3900, 3925, 3929, 3935, 3950, 3963, 3973, 3977, 4026, 4049, 4057, 4062, 4079, 4126, 4136, 4138, 4187, 4195, 4196, 4201, 4207, 4209, 4221}\\
\end{longtable}

\bibliographystyle{plain}
\bibliography{bibliography}

\begin{thebibliography}{10}

\bibitem{Pragmatic}
Mohammad Akhtar, Tom Coates, Alessio Corti, Liana Heuberger, Alexander~M.
  Kasprzyk, Alessandro Oneto, Andrea Petracci, Thomas Prince, and Ketil
  Tveiten.
\newblock Mirror symmetry and the classification of orbifold del {P}ezzo
  surfaces.
\newblock {\em Proc. Amer. Math. Soc.}, 144(2):513--527, 2016.

\bibitem{ACGK12}
Mohammad Akhtar, Tom Coates, Sergey Galkin, and Alexander~M. Kasprzyk.
\newblock Minkowski polynomials and mutations.
\newblock {\em SIGMA Symmetry Integrability Geom. Methods Appl.}, 8:Paper 094,
  17, 2012.

\bibitem{AK14}
Mohammad Akhtar and Alexander~M. Kasprzyk.
\newblock Singularity content.
\newblock \href{http://arxiv.org/abs/1401.5458}{\texttt{arXiv:1401.5458
  [math.AG]}}, 2014.

\bibitem{AK13}
Mohammad Akhtar and Alexander~M. Kasprzyk.
\newblock Mutations of fake weighted projective planes.
\newblock {\em Proc. Edinb. Math. Soc. (2)}, 59(2):271--285, 2016.

\bibitem{Aur07}
Denis Auroux.
\newblock Mirror symmetry and {$T$}-duality in the complement of an
  anticanonical divisor.
\newblock {\em J. G{\"o}kova Geom. Topol. GGT}, 1:51--91, 2007.

\bibitem{Bat81}
Victor~V. Batyrev.
\newblock Toric {F}ano threefolds.
\newblock {\em Izv. Akad. Nauk SSSR Ser. Mat.}, 45(4):704--717, 927, 1981.

\bibitem{Bat85}
Victor~V. Batyrev.
\newblock {\em Higher dimensional toric varieties with ample anticanonical
  class}.
\newblock PhD thesis, Moscow State University, 1985.
\newblock Text in Russian.

\bibitem{Bat04}
Victor~V. Batyrev.
\newblock Toric degenerations of {F}ano varieties and constructing mirror
  manifolds.
\newblock In {\em The {F}ano {C}onference}, pages 109--122. Univ. Torino,
  Turin, 2004.

\bibitem{BB96}
Victor~V. Batyrev and Lev~A. Borisov.
\newblock Dual cones and mirror symmetry for generalized {C}alabi-{Y}au
  manifolds.
\newblock In {\em Mirror symmetry, {II}}, volume~1 of {\em AMS/IP Stud. Adv.
  Math.}, pages 71--86. Amer. Math. Soc., Providence, RI, 1997.

\bibitem{BCFKS97}
Victor~V. Batyrev, Ionu{\c{t}} Ciocan-Fontanine, Bumsig Kim, and Duco van
  Straten.
\newblock Conifold transitions and mirror symmetry for {C}alabi-{Y}au complete
  intersections in {G}rassmannians.
\newblock {\em Nuclear Phys. B}, 514(3):640--666, 1998.

\bibitem{BCFKS98}
Victor~V. Batyrev, Ionu{\c{t}} Ciocan-Fontanine, Bumsig Kim, and Duco van
  Straten.
\newblock Mirror symmetry and toric degenerations of partial flag manifolds.
\newblock {\em Acta Math.}, 184(1):1--39, 2000.

\bibitem{BCFK03}
Aaron Bertram, Ionu{\c{t}} Ciocan-Fontanine, and Bumsig Kim.
\newblock Two proofs of a conjecture of {H}ori and {V}afa.
\newblock {\em Duke Math. J.}, 126(1):101--136, 2005.

\bibitem{magma}
Wieb Bosma, John Cannon, and Catherine Playoust.
\newblock The {M}agma algebra system. {I}. {T}he user language.
\newblock {\em J. Symbolic Comput.}, 24(3-4):235--265, 1997.
\newblock Computational algebra and number theory (London, 1993).

\bibitem{GRDb}
Gavin Brown and Alexander~M. Kasprzyk.
\newblock The {G}raded {R}ing {D}atabase.
\newblock Online.
\newblock Access via
  \href{http://www.grdb.co.uk/}{\texttt{http://www.grdb.co.uk/}}.

\bibitem{CKP13}
Ivan Cheltsov, Ludmil Katzarkov, and Victor Przyjalkowski.
\newblock Birational geometry via moduli spaces.
\newblock In {\em Birational geometry, rational curves, and arithmetic}, pages
  93--132. Springer, New York, 2013.

\bibitem{CPSh05}
Ivan Cheltsov, Victor Przyjalkowski, and Constantin Shramov.
\newblock Hyperelliptic and trigonal {F}ano threefolds.
\newblock {\em Izv. Ross. Akad. Nauk Ser. Mat.}, 69(2):145--204, 2005.

\bibitem{CI12}
Jan~Arthur Christophersen and Nathan Ilten.
\newblock Degenerations to unobstructed {F}ano {S}tanley-{R}eisner schemes.
\newblock {\em Math. Z.}, 278(1-2):131--148, 2014.

\bibitem{CI14}
Jan~Arthur Christophersen and Nathan Ilten.
\newblock Hilbert schemes and toric degenerations for low degree {F}ano
  threefolds.
\newblock {\em J. Reine Angew. Math.}, 717:77--100, 2016.

\bibitem{CCGGK13}
Tom Coates, Alessio Corti, Sergey Galkin, Vasily Golyshev, and Alexander~M.
  Kasprzyk.
\newblock Mirror symmetry and {F}ano manifolds.
\newblock In {\em Proceedings of the $6$th European Congress of Mathematics
  (Krak{\'o}w, 2-7 July, 2012)}, pages 285--300. European Mathematical Society,
  November 2013.

\bibitem{CCGK13}
Tom Coates, Alessio Corti, Sergey Galkin, and Alexander~M. Kasprzyk.
\newblock Quantum periods for 3-dimensional {F}ano manifolds.
\newblock {\em Geom. Topol.}, 20(1):103--256, 2016.

\bibitem{CKP14}
Tom Coates, Alexander~M. Kasprzyk, and Thomas Prince.
\newblock Four-dimensional {F}ano toric complete intersections.
\newblock {\em Proc. Roy. Soc. London Ser. A}, 471:20140704, 2015.

\bibitem{CKP}
Tom Coates, Alexander~M. Kasprzyk, and Thomas Prince.
\newblock Laurent inversion.
\newblock \href{http://arxiv.org/abs/1707.05842}{\texttt{arXiv:1707.05842
  [math.AG]}}, 2017.

\bibitem{Con00}
Brian Conrad.
\newblock {\em Grothendieck duality and base change}, volume 1750 of {\em
  Lecture Notes in Mathematics}.
\newblock Springer-Verlag, Berlin, 2000.

\bibitem{Dai02}
Dimitrios~I. Dais.
\newblock Resolving 3-dimensional toric singularities.
\newblock In {\em Geometry of toric varieties}, volume~6 of {\em S\'emin.
  Congr.}, pages 155--186. Soc. Math. France, Paris, 2002.

\bibitem{DKK13}
Collin Diemer, Ludmil Katzarkov, and Gabriel Kerr.
\newblock Compactifications of spaces of {L}andau-{G}inzburg models.
\newblock {\em Izv. Ross. Akad. Nauk Ser. Mat.}, 77(3):55--76, 2013.

\bibitem{DK}
George Dimitrov and Ludmil Katzarkov.
\newblock {Non-commutative counting invariants and curve complexes}.
\newblock arXiv:1805.00294.

\bibitem{DN77}
Igor~V. Dolgachev and Viacheslav~V. Nikulin.
\newblock Exceptional singularities of {V.I.} {A}rnold and {K}$3$ surfaces.
\newblock In {\em Proc. USSR Topological Conf. in Minsk}, 1977.

\bibitem{DH16}
Charles Doran and Andrew Harder.
\newblock {Toric Degenerations and the Laurent polynomials related to
  Givental's Landau-Ginzburg models}.
\newblock {\em Canad. J. Math.}, 68:784--815, 2016.

\bibitem{DHKLP}
Charles Doran, Andrew Harder, Ludmil Katzarkov, Jacob Lewis, and Victor
  Przyjalkowski.
\newblock Modularity of {F}ano threefolds.
\newblock In preparation.

\bibitem{EGH97}
Tohru Eguchi, Kentaro Hori, and Chuan-Sheng Xiong.
\newblock Gravitational quantum cohomology.
\newblock {\em Internat. J. Modern Phys. A}, 12(9):1743--1782, 1997.

\bibitem{Gal07}
Sergei Galkin.
\newblock Small toric degenerations of {F}ano $3$-folds.
\newblock Available at {{\url{http://member.ipmu.jp/sergey.galkin/}}}, 2007.

\bibitem{GU10}
Sergey Galkin and Alexandr Usnich.
\newblock Mutations of potentials.
\newblock Preprint IPMU 10-0100, 2010.

\bibitem{Gi97}
Alexander Givental.
\newblock A mirror theorem for toric complete intersections.
\newblock In {\em Topological field theory, primitive forms and related topics
  ({K}yoto, 1996)}, volume 160 of {\em Progr. Math.}, pages 141--175.
  Birkh\"{a}user Boston, Boston, MA, 1998.

\bibitem{GOL}
Vasily Golyshev.
\newblock {Deresonating a Tate period}.
\newblock arXiv:0908.1458.

\bibitem{M2}
Daniel~R. Grayson and Michael~E. Stillman.
\newblock Macaulay2, a software system for research in algebraic geometry.
\newblock Available at \url{http://www.math.uiuc.edu/Macaulay2/}.

\bibitem{GHK13}
Mark Gross, Paul Hacking, and Sean Keel.
\newblock Birational geometry of cluster algebras.
\newblock {\em Algebr. Geom.}, 2(2):137--175, 2015.

\bibitem{HP10}
Paul Hacking and Yuri Prokhorov.
\newblock Smoothable del~{P}ezzo surfaces with quotient singularities.
\newblock {\em Compos. Math.}, 146(1):169--192, 2010.

\bibitem{HKKP}
Fabian Haiden, Ludmil Katzarrov, Maxim Kontsevich, and Pranav Pandit.
\newblock {Categorical Kaehler geometry}.
\newblock in preparation.

\bibitem{harder}
Andrew Harder and Charles Doran.
\newblock Toric degenerations and the {L}aurent polynomials related to
  {G}ivental's {L}andau-{G}inzburg models.
\newblock \href{http://arxiv.org/abs/1502.02079}{\texttt{arXiv:1502.02079
  [math.AG]}}, 2015.

\bibitem{HV00}
Kentaro Hori and Cumrun Vafa.
\newblock Mirror symmetry.
\newblock
  \href{http://arxiv.org/abs/hep-th/0002222}{\texttt{arXiv:hep-th/0002222v3}},
  2000.

\bibitem{Ilt12}
Nathan~Owen Ilten.
\newblock Mutations of {L}aurent polynomials and flat families with toric
  fibers.
\newblock {\em SIGMA Symmetry Integrability Geom. Methods Appl.}, 8:Paper 047,
  7, 2012.

\bibitem{ILP11}
Nathan~Owen Ilten, Jacob Lewis, and Victor Przyjalkowski.
\newblock Toric degenerations of {F}ano threefolds giving weak
  {L}andau-{G}inzburg models.
\newblock {\em J. Algebra}, 374:104--121, 2013.

\bibitem{Is77}
V.~A. Iskovskih.
\newblock Fano threefolds. {I}.
\newblock {\em Izv. Akad. Nauk SSSR Ser. Mat.}, 41(3):516--562, 717, 1977.

\bibitem{Is78}
V.~A. Iskovskih.
\newblock Fano threefolds. {II}.
\newblock {\em Izv. Akad. Nauk SSSR Ser. Mat.}, 42(3):506--549, 1978.

\bibitem{IM71}
V.~A. Iskovskih and Ju.~I. Manin.
\newblock Three-dimensional quartics and counterexamples to the {L}{\"u}roth
  problem.
\newblock {\em Mat. Sb. (N.S.)}, 86(128):140--166, 1971.

\bibitem{IP99}
Vasily Iskovskikh and Yury Prokhorov.
\newblock {F}ano varieties.
\newblock In {\em Algebraic geometry, {V}}, volume~47 of {\em Encyclopaedia
  Math. Sci.}, pages 1--247. Springer, Berlin, 1999.

\bibitem{JR06}
Priska Jahnke and Ivo Radloff.
\newblock Gorenstein {F}ano threefolds with base points in the anticanonical
  system.
\newblock {\em Compos. Math.}, 142(2):422--432, 2006.

\bibitem{Ka09}
Ilia Karzhemanov.
\newblock On {F}ano threefolds with canonical {G}orenstein singularities.
\newblock {\em Mat. Sb.}, 200(8):111--146, 2009.

\bibitem{KNP14}
Alexander~M. Kasprzyk, Benjamin Nill, and Thomas Prince.
\newblock Minimality and mutation-equivalence of polygons.
\newblock {\em Forum Math. Sigma}, 5:e18, 48, 2017.

\bibitem{KP12}
Ludmil Katzarkov and Victor Przyjalkowski.
\newblock Landau-{G}inzburg models---old and new.
\newblock In {\em Proceedings of the {G}{\"o}kova {G}eometry-{T}opology
  {C}onference 2011}, pages 97--124. Int. Press, Somerville, MA, 2012.

\bibitem{Klep79}
Jan~O. Kleppe.
\newblock Deformations of graded algebras.
\newblock {\em Math. Scand.}, 45(2):205--231, 1979.

\bibitem{KS-B88}
J.~Koll{\'a}r and N.~I. Shepherd-Barron.
\newblock Threefolds and deformations of surface singularities.
\newblock {\em Invent. Math.}, 91(2):299--338, 1988.

\bibitem{KS98}
Maximilian Kreuzer and Harald Skarke.
\newblock Classification of reflexive polyhedra in three dimensions.
\newblock {\em Adv. Theor. Math. Phys.}, 2(4):853--871, 1998.

\bibitem{KS00}
Maximilian Kreuzer and Harald Skarke.
\newblock Complete classification of reflexive polyhedra in four dimensions.
\newblock {\em Adv. Theor. Math. Phys.}, 4(6):1209--1230, 2000.

\bibitem{KS04}
Maximilian Kreuzer and Harald Skarke.
\newblock {PALP}, a package for analyzing lattice polytopes with applications
  to toric geometry.
\newblock {\em Computer Phys. Comm.}, 157:87--106, 2004.

\bibitem{Ma99}
Yuri {Manin}.
\newblock {\em Frobenius manifolds, quantum cohomology, and moduli spaces.}
\newblock Providence, RI: American Mathematical Society, 1999.

\bibitem{MM81}
Shigefumi Mori and Shigeru Mukai.
\newblock Classification of {F}ano {$3$}-folds with {$B_{2}\geq 2$}.
\newblock {\em Manuscripta Math.}, 36(2):147--162, 1981/82.

\bibitem{MM03}
Shigefumi Mori and Shigeru Mukai.
\newblock Erratum: ``{C}lassification of {F}ano 3-folds with {$B_2\geq 2$}''
  [{M}anuscripta {M}ath. {\bf 36} (1981/82), no. 2, 147--162].
\newblock {\em Manuscripta Math.}, 110(3):407, 2003.

\bibitem{Mu02}
Shigeru Mukai.
\newblock New developments in the theory of {F}ano threefolds: vector bundle
  method and moduli problems [translation of {S}\=ugaku {\bf 47} (1995), no.\
  2, 125--144].
\newblock {\em Sugaku Expositions}, 15(2):125--150, 2002.
\newblock Sugaku expositions.

\bibitem{Pin77}
Henry Pinkham.
\newblock Singularit\'es exceptionnelles, la dualit\'e \'etrange d'{A}rnold et
  les surfaces {$K-3$}.
\newblock {\em C. R. Acad. Sci. Paris S\'er. A-B}, 284(11):A615--A618, 1977.

\bibitem{Pr19}
Thomas Prince.
\newblock {From Cracked Polytopes to Fano Threefolds}.
\newblock arXiv:1904.01077.

\bibitem{Pri15}
Thomas Prince.
\newblock Smoothing toric {F}ano surfaces using the {G}ross--{S}iebert
  algorithm.
\newblock \href{http://arxiv.org/abs/1504.05969}{\texttt{arXiv:1504.05969
  [math.AG]}}, 2015.

\bibitem{Pr05}
Yury Prokhorov.
\newblock The degree of {F}ano threefolds with canonical {G}orenstein
  singularities.
\newblock {\em Sb. Math.}, 196(1):77--114, 2005.

\bibitem{Prz07}
Victor Przyjalkowski.
\newblock On {L}andau--{G}inzburg models for {F}ano varieties.
\newblock {\em Commun. Number Theory Phys.}, 1(4):713--728, 2007.

\bibitem{Prz09}
Victor Przyjalkowski.
\newblock Weak {L}andau-{G}inzburg models of smooth {F}ano threefolds.
\newblock {\em Izv. Ross. Akad. Nauk Ser. Mat.}, 77(4):135--160, 2013.

\bibitem{Prz17}
Victor Przyjalkowski.
\newblock {Calabi--Yau compactifications of toric Landau--Ginzburg models for
  smooth Fano threefolds}.
\newblock {\em Sb. Math.}, 208(7):992--1013, 2017.

\bibitem{Prz18}
Victor Przyjalkowski.
\newblock {Toric Landau--Ginzburg models}.
\newblock {\em Russ. Mat. Surv.}, 73(6):1033--1118, 2018.

\bibitem{Rab89}
Stanley Rabinowitz.
\newblock A census of convex lattice polygons with at most one interior lattice
  point.
\newblock {\em Ars Combin.}, 28:83--96, 1989.

\bibitem{TOPCOM}
J{\"o}rg Rambau.
\newblock {TOPCOM}: Triangulations of point configurations and oriented
  matroids.
\newblock In Arjeh~M. Cohen, Xiao-Shan Gao, and Nobuki Takayama, editors, {\em
  Mathematical Software---ICMS 2002}, pages 330--340. World Scientific, 2002.

\bibitem{Rei80}
Miles Reid.
\newblock Canonical {$3$}-folds.
\newblock In {\em Journ\'ees de G\'eometrie Alg\'ebrique d'Angers, Juillet
  1979/Algebraic Geometry, Angers, 1979}, pages 273--310. Sijthoff \&
  Noordhoff, Alphen aan den Rijn, 1980.

\bibitem{Sat00}
Hiroshi Sato.
\newblock Toward the classification of higher-dimensional toric {F}ano
  varieties.
\newblock {\em Tohoku Math. J. (2)}, 52(3):383--413, 2000.

\bibitem{Wah80}
Jonathan~M. Wahl.
\newblock Elliptic deformations of minimally elliptic singularities.
\newblock {\em Math. Ann.}, 253(3):241--262, 1980.

\bibitem{WW82}
K.~Watanabe and M.~Watanabe.
\newblock The classification of {F}ano $3$-folds with torus embeddings.
\newblock {\em Tokyo J. Math.}, 5(1):37--48, 1982.

\end{thebibliography}
\end{document}